\documentclass{aims}

\usepackage{txfonts}
\usepackage{enumerate}
\usepackage{amssymb,amsmath,amsfonts,amsthm,mathrsfs}
\usepackage{graphicx,color}
\usepackage{cite}
\usepackage{indentfirst}
\numberwithin{equation}{section}
\def\typeofarticle{Research article}
\def\currentvolume{10}
\def\currentissue{1}
\def\currentyear{2025}

\def\ppages{858-883}
\def\DOI{10.3934/math.2025041}
\def\Received{03 October 2024}
\def\Revised{06 January 2025}
\def\Accepted{10 January 2025}
\def\Published{15 January 2025}
\newtheorem{defn}{Definition}[section]
\newtheorem{thm}{Theorem}[section]

\newtheorem{cor}{Corollary}[section]

\usepackage{cite}
\numberwithin{equation}{section}

\makeatletter
\renewcommand{\@biblabel}[1]{#1\hfill \hspace{-0.2cm}}
\makeatother

\begin{document}

\title{The Weighted $\boldsymbol{L}^{\boldsymbol{p}}$ estimates for the fractional Hardy operator and a class of integral operators on the Heisenberg group}

\author{%
  Tianyang He\affil{1},
  Zhiwen Liu\affil{2},
  and
  Ting Yu\affil{1,}\corrauth
}

% \shortauthors is used in copyright information in the end of the paper
\shortauthors{the Author(s)}

\address{%
  \addr{\affilnum{1}}{Research Center for Mathematics and Interdisciplinary Sciences, Frontiers Science Center for Nonlinear Expectations (Ministry of Education), Shandong University, Qingdao, 266237, China}
  \addr{\affilnum{2}}{School of Science, Shandong Jianzhu University, Jinan, 250100, China}}

% corresponding author
\corraddr{Email: tingy@sdu.edu.cn.}

\begin{abstract}
In the setting of a Heisenberg group, we first studied the sharp weak estimate for the $n$-dimensional fractional Hardy operator from $L^p$ to $L^{q,\infty}$. Next, we studied the sharp bounds for the $m$-linear $n$-dimensional integral operator with a kernel on weighted Lebesgue spaces. As an application, the sharp bounds for Hardy, Hardy-Littlewood-P\'{o}lya, and Hilbert operators on weighted Lebesgue spaces were obtained. Finally, according to the previous steps, we also found the estimate for the Hausdorff operator on weighted $L^p$ spaces.
\end{abstract}


\maketitle

\section{Introduction}

It is well-known that averaging operators play an important role in harmonic
analysis. It is often desirable to obtain sharp norm estimates for them in
different function spaces. Our starting point is the Hardy operator and its duality
form:
$$Hf(x)=\frac{1}{x}\int_0^x{f(t)\mathrm{d}t,\quad H^*f(x)}=\int_x^{\infty}{\frac{f(t)}{t}\mathrm{d}t,}
$$
where $x>0$. Hardy \cite{Hardy} established the well-known Hardy integral inequalities
$$
\int_0^{\infty}{|Hf(x)|^p\mathrm{d}x}\leqslant \bigl( \frac{p}{p-1} \bigr) ^p\int_0^{\infty}{|f(x)|^p\mathrm{d}x,\quad p}>1,
$$
and
$$
\int_0^{\infty}{|H^*f(x)|^{p^{\prime}}\mathrm{d}x}\leqslant \bigl( \frac{p}{p-1} \bigr) ^{p^{\prime}}\int_0^{\infty}{|f(x)|^{p^{\prime}}\mathrm{d}x,\quad p}>1,
$$
where $p^{\prime}=p/(p-1)$. He proved that the constant $p/(p-1)$ is sharp. The
corresponding higher-dimensional Hardy operator was introduced by Faris \cite{BI}
in his study of quantum mechanics. Christ and Grafakos \cite{CH} gave the following
equivalent definition of $n$-dimensional Hardy operators:
$$
\mathscr{H} _nf(x)=\frac{1}{|x|^n}\int_{|y|<|x|}{f(y)\mathrm{d}y,}
$$
where $x\in \mathbb{R} ^n\setminus \{\theta \}$, and $\theta$ is the origin in $\mathbb{R}^n$. They showed that
$$
\parallel \mathscr{H} _n\parallel _{L^p(\mathbb{R} ^n)\rightarrow L^p(\mathbb{R} ^n)}=\frac{p}{p-1}\frac{\omega _n}{n},\quad 1<p<\infty .
$$
Here $\omega _n$ is the superficial area of the unit ball in $\mathbb{R}^n$. The Lebesgue spaces with power weights are another kind of function space to consider the sharp estimates
of the Hardy operator. The method in \cite{CH} is invalid in this case. Fu et al. \cite{Hardy1}, by the method of rotation, established the following estimate:
$$
\parallel \mathscr{H} _n\parallel _{L_{|x|^{\beta}}^{p}(\mathbb{R} ^n)\rightarrow L_{|x|^{\beta}}^{p}(\mathbb{R} ^n)}=\frac{\omega _n}{\frac{n}{p^{\prime}}-\frac{\beta}{p}}.
$$
For more information about the Hardy operator, we refer to the reader to \cite{LU}.

Meanwhile, the fractional Hardy operator is also very interesting since it is a useful tool to study the embedding properties of function spaces. Mizuta et al. \cite{fractional} showed that the optimal bound of the fractional Hardy operator implies the sharp embedding properties of function spaces. There is much literature on function spaces. Let $f$ be a locally integrable function on $\mathbb{R} ^n$. Then the $n$-dimensional fractional Hardy operator and its duality form are
$$H_{\alpha}f\left( x \right) =\frac{1}{\left| x \right|^{n-\alpha}}\int_{\left| y \right|<\left| x \right|}{f\left( y \right)}dy,\quad H_{\alpha}^{*}f\left( x \right) =\int_{\left| y \right|>\left| x \right|}{\frac{1}{\left| y \right|^{n-\alpha}}f\left( y \right)}dy,$$
where $0<\alpha<n$, $x\in \mathbb{R} ^n\backslash\left\{ 0\right\}$. If $\alpha$ = 0, the fractional Hardy operator is the classic Hardy operator. There is much literature on the boundness of these operators \cite{Hardyjieshao1,Hardyjieshao2,Hardyjieshao3,Hardyjieshao4}. Among them, Lu et al. \cite{Hardyjieshao2} obtained the following estimates. Suppose
$$0<\alpha <n,\quad1<p\leqslant \frac{n}{\alpha},\quad\frac{1}{p}-\frac{1}{q}=\frac{\alpha}{n}.$$
Then
$$\left\| H_{\alpha}f \right\| _{L^q\left( \mathbb{R} ^n \right)}\leqslant C\left\| f \right\| _{L^p\left( \mathbb{R} ^n \right)},$$
where
$$\left( \frac{p}{q} \right) ^{1/q}\left( \frac{p}{p-1} \right) ^{1/q}\left( \frac{q}{q-1} \right) ^{1-1/q}\left( 1-\frac{p}{q} \right) ^{1/p-1/q}\left( \frac{\omega _n}{n} \right) ^{1-\alpha /n}\leqslant C\leqslant \left( \frac{p}{p-1} \right) ^{p/q}\left( \frac{\omega _n}{n} \right) ^{1-\alpha /n}.
$$
If $p=1$, then
$$\left\| H_{\alpha} \right\| _{L^1\left( \mathbb{R} ^n \right) \rightarrow L^{n/\left( n-\alpha \right) ,\infty}\left( \mathbb{R} ^n \right)}=\left( \frac{\omega _n}{n} \right) ^{1-\frac{\alpha}{n}}.$$
The optimal $L^p\rightarrow L^q$ estimate was later obtained in \cite{Hardyjieshao3}:
$$\left\| H_{\alpha} \right\| _{L^p\left( \mathbb{R} ^n \right) \rightarrow L^q\left( \mathbb{R} ^n \right)}=\left( \frac{p\prime}{q} \right) ^{1/q}\left( \frac{n}{q\alpha}\cdot B\left( \frac{n}{q\alpha},\frac{n}{q^{\prime}\alpha} \right) \right) ^{-\alpha /n}\left( \frac{\omega _n}{n} \right) ^{1-\frac{\alpha}{n}},$$
where 
$$p^{\prime}=\frac{p}{p-1},\quad q^{\prime}=\frac{q}{q-1},$$
and $B\left( \cdot ,\cdot \right)$ is the beta function defined by
$$B\left( z,\omega \right) =\int_0^1{t^{z-1}\left( 1-t \right) ^{\omega -1}dt},$$
where $z$ and $\omega$ are complex numbers with positive real parts. Comparing with
the complicated bounds in the power weighted spaces, the sharp weak bounds for $H_{\alpha}$ and $H_{\alpha}^{*}$ seem easier to understand. Gao and Zhao \cite{Hardyjieshao4} set up
$$
\left\| H_{\alpha}^{*} \right\| _{L^1(\mathbb{R} ^n)\rightarrow L^{n/(n-\alpha ),\infty}(\mathbb{R} ^n)}=\left( \frac{\omega _n}{n} \right) ^{1-\frac{\alpha}{n}},\quad\parallel H_{\alpha}^{*}\parallel _{L^p(\mathbb{R} ^n)\rightarrow L^{q,\infty}(\mathbb{R} ^n)}\,=\left( \frac{\omega _n}{n} \right) ^{\frac{1}{q}+\frac{1}{p^{\prime}}}\bigl( \frac{q}{p^{\prime}} \bigr) ^{1/p^{\prime}}.
$$

It is then a nature problem to obtain the operator norms of $H_\alpha$ and its dual operator $H_\alpha^*$ in corresponding power weighted spaces. For the latter, Gao et al. \cite{GA} set up
$$
\left\|H_\alpha^*\right\|_{L_{|x| \rho}^1\left(\mathbb{R}^n\right) \rightarrow L_{|x|^\beta}^{(n+\beta) /(n-\alpha+\rho), \infty}\left(\mathbb{R}^n\right)}=\left(\frac{\omega_n}{n+\beta}\right)^{(n-\alpha+\rho) /(n+\beta)},
$$
and
$$
\left\|H_\alpha^*\right\|_{L_{|x|}^p \rho}\left(\mathbb{R}^n\right) \rightarrow L_{|x| \beta}^{q, \infty}\left(\mathbb{R}^n\right)=\left(\frac{\omega_n}{n+\beta}\right)^{\frac{1}{q}+\frac{1}{p^{\prime}}}\left(\frac{q}{p^{\prime}}\right)^{1 / p^{\prime}}.
$$
Then Yu et al. \cite{fraction} set up
$$
\parallel H_{\alpha}\parallel _{L_{|x|^{\rho}}^{p}(\mathbb{R} ^n)\rightarrow L_{|x|^{\beta}}^{q,\infty}(\mathbb{R} ^n)}\,=\bigl( \frac{\omega _n}{n+\beta} \bigr) ^{1/q}\biggl( \frac{\omega _n}{n-\frac{\rho}{p-1}} \biggr) ^{1/p^{\prime}},
$$
and
$$
\parallel H_{\alpha}\parallel _{L^1(\mathbb{R} ^n)\rightarrow L_{|x|^{\beta}}^{(n+\beta )/(n-\alpha ),\infty}(\mathbb{R} ^n)}\,=\bigl( \frac{\omega _n}{n+\beta} \bigr) ^{(n-\alpha )/(n+\beta )}.
$$

Inspired by them, we will study the sharp weak bound for the fractional Hardy operator on the Heisenberg group, which plays an important role in several branches of mathematics. Now, allow us to introduce some basic knowledge about the Heisenberg group which will be used later.

The Heisenberg group is a very typical non-commutative group, and research on up-modulation and analytic problems is an extension of Euclidean space upharmonic and analytical problems, which is an important part of non-commutative harmonic analysis \cite{cou, tha}. Harmonic analysis on the Heisenberg group has been drawing more and more attention, see\cite{chu} and \cite{zha}.

Let us introduce some basic knowledge about the Heisenberg group. The Heisenberg group $\mathbb{H}^n$ is a non-commutative nilpotent Lie group, with the underlying manifold $\mathbb{R}^{2n}\times\mathbb{R}$ and the group law.

Let
$$
x=(x_1,\dots,x_{2n},x_{2n+1}),\quad y=(y_1,\dots,y_{2n},y_{2n+1}),
$$
and then
$$
x\times y=(x_1+y_1,\dots,x_{2n}+y_{2n},x_{2n+1}+y_{2n+1}+2\sum_{j=1}^n(y_jx_{n+j}-x_jy_{n+j}).
$$

The Heisenberg group $\mathbb{H}^n$ is a homogeneous group with dilations
$$
\delta_r(x_1,x_2,\dots,x_{2n},x_{2n+1})=(rx_1,rx_2,\dots,rx_{2n},r^2x_{2n+1}),\quad  r>0.
$$

The Haar measure on $\mathbb{H}^n$ coincides with the usual Lebesgue measure on $\mathbb{R}^{2n+1}$. We denote any measurable set $E\subset \mathbb{H}^n$ by $\left|E\right|$, and then
$$
\left| \delta_r \left(E\right) \right|=r^Q\left|E\right|,\quad d(\delta_rx)=r^Qdx,
$$
where $Q=2n+2$ is called the homogeneous dimension of $\mathbb{H}^n$.

The Heisenberg distance derived from the norm
$$
| x |_h=\left[ \left( \sum_{i=1}^{2n}{x_{i}^{2}} \right) ^2+x_{2n+1}^{2} \right]^{{{1}/{4}}},
$$
where $x=(x_1,x_2,\dots,x_{2n},x_{2n+1})$, is given by
$$
d(p,q)=d(q^{-1}p,0)=|q^{-1}p|_h.
$$
This distance $d$ is left-invariant, it meaning that $d(p,q)$ remains constant when both $p$ and $q$ are left shifted by some fixed vector on $\mathbb{H}^n$. Furthermore, $d$ satisfies the trigonometric inequality defined by \cite{kor}:
$$
d(p,q)\leq {d(p,x)+d(x,q)},\quad p,x,q\in \mathbb{H}^n.
$$
For $r>0$ and $x\in \mathbb{H}^n$, the ball and sphere with center $x$ and radius $r$ on $\mathbb{H}^n$ are given by
$$
B\left( x,r \right) =\left\{ y\in \mathbb{H}^n:d\left( x,y \right) <r \right\},
$$
and
$$
S\left( x,r \right) =\left\{ y\in \mathbb{H}^n:d\left( x,y \right) =r \right\} .
$$
Then we obtain
$$
|B(x,r)|=|B(0,r)|=\varOmega_Qr^Q,
$$
where
$$
\varOmega_Q=\frac{2\pi^{n+\frac{1}{2}}\varGamma \left({{n}/{2}} \right)}{(n+1) \varGamma
	(n) \varGamma((n+1)/2)}
$$
represents the volume of the unit sphere $B(0,1)$ on $\mathbb{H}^n$, and $\omega _Q=Q\varOmega _Q$. The reader is referred to \cite{tha,lunwenji2} for more details.

Except for the fractional Hardy operator, we also study the sharp bounds for some $m$-linear $n$-dimensional integral operators on the Heisenberg group. In 2017, Batbold and Sawano \cite{HLP} studied one-dimensional $m$-linear Hilbert-type operators that include the Hardy-Littlewood-P\'{o}lya operator on weighted Morrey spaces, and they obtained the sharp bounds. He et al. \cite{HLPtuiguang} extended the results in \cite{HLP} and obtained the sharp bound for the generalized Hardy-Littlewood-P\'{o}lya operator on weighted central and noncentral homogenous Morrey spaces. 
He set up
$$
\parallel T(f_1,\cdots ,f_m)\parallel _{L^{q,\lambda}(\mathbb{R} ^n,|x|^{\alpha},|x|^{\gamma})}\,\le C_m\prod_{j=1}^m{\parallel f_j}\parallel _{L^{q_j,\lambda}(\mathbb{R} ^n,|x|^{\alpha},|x|^{\frac{q_j\gamma _j}{q}})},
$$
where 
$$
C_m=\int_{\mathbb{R} ^{nm}}{K(y_1,}\cdots ,y_m)\prod_{i=1}^m{|y_i|^{-d(\lambda _i,q_i,\alpha ,\frac{q_i\gamma _i}{q})}dy_1}\cdots dy_m<\infty .
$$

In 2011, Wu and Fu \cite{p-adicHardy} got the sharp estimate of the $m$-linear $p$-adic Hardy operator on Lebesgue spaces with power weights. Zhang et al \cite{Hausdorff} obtained the sharp estimate for the $m$-linear $n$-dimensional Hausdorff operator on the weighted Morrey space.

Inspired by the above, on the Heisenberg group, we first study the sharp weak estimate for the $n$-dimensional fractional Hardy operator from $L^p$ to $L^{q,\infty}$. Second, we study a more general operator which includes Hardy, Hardy-Littlewood-P\'{o}lya, and Hilbert operators as a special case and consider their operator norm on weighted Lebesgue space. Finally, we also find the sharp bound for the Hausdorff operator on Lebesgue space, which generalizes the previous results.

To get the main conclusion, it is necessary to introduce some fundamental knowledge and
definitions. In the setting of the Heisenberg group, these operators and spaces are the fractional Hardy operator, $m$-linear $n$-dimensional Hardy operator, $m$-linear $n$-dimensional Hardy-Littlewood-P\'{o}lya operator, $m$-linear $n$-dimensional Hilbert operator, $m$-linear $n$-dimensional Hausdorff operator, and weighted $L^p$ and $L^{q,\infty}$.
\begin{defn}
	Let $f$ be nonnegative locally integrable functions on $\mathbb{H}^n$ and $Q=2n+1$, $0<\alpha <Q$. The $n$-dimensional fractional Hardy operator is defined by
	\begin{align}
		\mathcal{H} _{\alpha}f\left( x \right)= \frac{1}{|x|_{h}^{Q-\alpha}}\int_{|y|_h<|x|_h}{f\left( y \right) dy},
	\end{align}
	where $x\in \mathbb{H}^n\backslash\left\{ 0\right\}$.
\end{defn}
\begin{defn}
	Let $m$ be a positive integer and $f_1,\dots,f_m$ be nonnegative locally integrable functions on $\mathbb{H}^n$. The $m$-linear $n$-dimensional Hardy operator is defined by
	\begin{align}
		\mathcal{H} _{1}^{h}(f_1,...,f_m)(x)=\frac{1}{|x|_{h}^{mQ}}\int_{|(y_1,...,y_m)|_h\leqslant |x|_h}{f_1(y_1)\cdots f_m(y_m)dy_1\cdots dy_m},
	\end{align}
	where $x\in \mathbb{H}^n\backslash\left\{ 0\right\}$.
\end{defn}
\begin{defn}
	Let $m$ be a positive integer and $f_1,\dots,f_m$ be nonnegative locally integrable
	functions on $\mathbb{H}^n$. The $m$-linear $n$-dimensional Hardy-Littlewood-P\'{o}lya operator  is defined by
	\begin{align}
		\mathcal{H} _{2}^{h}(f_1,...,f_m)(x)=\int_{\mathbb{H} ^n}{\cdots}\int_{\mathbb{H} ^n}{\frac{f_1(y_1)\cdots f_m(y_m)}{[\max\mathrm{(}|x|_{h}^{Q},|y_1|_{h}^{Q},...,|y_m|_{h}^{Q})]^m}}dy_1\cdots dy_m,
	\end{align}
	where $x\in \mathbb{H}^n\backslash\left\{ 0\right\}$.
\end{defn}
\begin{defn}
	Let $m$ be a positive integer and $f_1,\dots,f_m$ be nonnegative locally integrable functions on $\mathbb{H}^n$. The $m$-linear $n$-dimensional Hilbert operator is defined by
	\begin{align}
		\mathcal{H} _{3}^{h}(f_1,...,f_m)(x)=\int_{\mathbb{H} ^n}{\cdots}\int_{\mathbb{H} ^n}{\frac{f_1(y_1)\cdots f_m(y_m)}{(|x|_{h}^{Q}+|y_1|_{h}^{Q}+\cdots +|y_m|_{h}^{Q})^m}}dy_1\cdots dy_m,
	\end{align}
	where $x\in \mathbb{H}^n\backslash\left\{ 0\right\}$.
\end{defn}
\begin{defn}
	Let $m$ be a positive integer, $f_1,\dots,f_m$ be nonnegative locally integrable functions on $\mathbb{H}^n$ and $\Phi$ be a nonnegative function on the Heisenberg group. The $m$-linear $n$-dimensional Hausdorff operator is defined by
	\begin{align}
		\mathcal{H} _{\Phi}^{h}(f_1,...,f_m)(x)=\int_{\mathbb{H} ^n}{\cdots}\int_{\mathbb{H} ^n}{\frac{\Phi (\delta _{\left| y_1 \right|_{h}^{-1}}x,...,\delta _{\left| y_m \right|_{h}^{-1}}x)}{|y_1|_{h}^{Q}\cdots |y_m|_{h}^{Q}}}f_1(y_1)\cdots f_m(y_m)dy_1\cdots dy_m,
	\end{align}
	where $x\in \mathbb{H}^n\backslash\left\{ 0\right\}$.
\end{defn}
\begin{defn}
	Let $1\leqslant p<\infty$. The Lebesgue space on the Heisenberg $L^p\left( \mathbb{H} ^n\right)$ is defined by
	$$L^p\left( \mathbb{H} ^n\right) =\{ f\in L_{loc}^{p}:\left\| f \right\| _{L^p\left( \mathbb{H} ^n \right)}<\infty \},$$
	where
	\begin{align}
		\left\| f \right\| _{L^p\left( \mathbb{H} ^n \right)}=\left( \int_{\mathbb{H} ^n}{\left| f\left( x \right) \right|^p dx} \right) ^{1/p}.
	\end{align}
\end{defn}
\begin{defn}
	Let $\omega : \mathbb{H} ^n\rightarrow \left( 0,\infty \right)$ be a positive measurable function, $1\leqslant p<\infty$. The weighted Lebesgue space on the Heisenberg $L^p\left( \mathbb{H} ^n,\omega \right)$ is defined by
	$$L^p\left( \mathbb{H} ^n,\omega \right) =\{ f\in L_{loc}^{p}:\left\| f \right\| _{L^p\left( \mathbb{H} ^n,\omega \right)}<\infty \},$$
	where
	\begin{align}
		\left\| f \right\| _{L^p\left( \mathbb{H} ^n,\omega \right)}=\left( \int_{\mathbb{H} ^n}{\left| f\left( x \right) \right|^p\omega \left( x \right) dx} \right) ^{1/p}.
	\end{align}
\end{defn}
\begin{defn}
	Let $\omega : \mathbb{H} ^n\rightarrow \left( 0,\infty \right)$ be a positive measurable function, $1\leqslant p<\infty$. The weighted weak Lebesgue space on the Heisenberg $L^p\left( \mathbb{H} ^n,\omega \right)$ is defined by
	$$L^{q,\infty}\left( \mathbb{H} ^n,\omega \right) =\left\{ f\in L_{loc}^{p}:\left\| f \right\| _{L^{q,\infty}\left( \mathbb{H} ^n,\omega \right)}<\infty \right\} ,$$
	where
	\begin{align}
		\left\| f \right\| _{L^{q,\infty}\left( \mathbb{H} ^n,\omega \right)}=\underset{\lambda >0}{\mathrm{sup}}\,\lambda \left( \int_{\mathbb{H} ^n}{\chi _{\left\{ x:f\left( x \right) >\lambda \right\}}\left( x \right) \omega \left( x \right) dx} \right) ^{1/q}.
	\end{align}
\end{defn}

Next, we will provide the main results of this article.

\section{Sharp estimate for the fractional Hardy operator}
In this section, we will study the weighted $L^p$ estimate for the fractional Hardy operator on the Heisenberg group. For the $n$-dimensional fractional Hardy operator, our results have a restricted condition: $\beta =0$ when $p=1$ and $\beta>0$ when $p>1$. Removing this restrictive condition requires a more complicated argument, and it will be presented in a future paper..
\begin{thm}\label{thm2.1}
	Let $1<p<\infty$, $1\leqslant q<\infty$, $\beta <Q\left( p-1 \right)$, $Q+\gamma >0$, and  $0\leqslant\alpha <\frac{\beta}{p-1}$.
	
	\noindent If
	$$\frac{\gamma +Q}{q}+\alpha =\frac{\beta +Q}{p},$$
	
	\noindent then
	\begin{align}
		\left\| \mathcal{H}_{\alpha} \right\| _{L^p( \mathbb{H} ^n,\left| x \right|_{h}^{\beta} ) \rightarrow L^{q,\infty}( \mathbb{H} ^n,\left| x \right|_{h}^{\gamma} )}=\left( \frac{\omega _Q}{Q+\gamma} \right) ^{\frac{1}{q}}\left( \frac{\omega _Q\left( p-1 \right)}{pQ-Q-\beta} \right) ^{\frac{1}{p^{\prime}}}.
	\end{align}
	
\end{thm}
\begin{thm}\label{thm2.2}
	Let $Q+\beta >0$ and $0<\alpha <Q$. Then
	\begin{align}\left\| \mathcal{H}_{\alpha} \right\| _{L^1\left( \mathbb{H} ^n \right) \rightarrow L^{\left( Q+\beta \right) /\left( Q-\alpha \right) ,\infty}( \mathbb{H} ^n,\left| x \right|_{h}^{\gamma} )}=\left( \frac{\omega _Q}{Q+\beta} \right) ^{\left( Q-\alpha \right) /\left( Q+\beta \right)}.
	\end{align}
\end{thm}
\begin{proof}[Proof of Theorem \ref{thm2.1}.]\renewcommand{\qedsymbol}{}
	Noticing $Q-\frac{\beta}{p-1}>Q-\frac{Q\left( p-1 \right)}{p-1}=0$, by $\mathrm{H}\ddot{\mathrm{o}}\mathrm{lder}$'s inequality, we have
	\begin{align*}
		|\mathcal{H} _{\alpha}f\left( x \right) |
		&=\left| \frac{1}{|x|_{h}^{Q-\alpha}}\int_{|y|_h<|x|_h}{f\left( y \right) dy} \right|=\left| \frac{1}{|x|_{h}^{Q-\alpha}}\int_{|y|_h<|x|_h}{|y|_{h}^{-\frac{\beta}{p}}f\left( y \right) |y|_{h}^{\frac{\beta}{p}}dy} \right|
		\\
		&\leqslant \frac{1}{|x|_{h}^{Q-\alpha}}\left( \int_{|y|_h<|x|_h}{|y|_{h}^{-\frac{\beta p\prime}{p}}dy} \right) ^{\frac{1}{p^{\prime}}}\left( \int_{|y|_h<|x|_h}{\left| f\left( y \right) \right|^p\left| y \right|_{h}^{\beta}dy} \right) ^{\frac{1}{p}}
		\\
		&\leqslant \frac{1}{|x|_{h}^{Q-\alpha}}\left( \omega _Q\int_0^{\left| x \right|_h}{r ^{-\frac{\beta}{p}\times \frac{p}{p-1}+Q-1}dr} \right) ^{\frac{1}{p^{\prime}}}\left( \int_{\mathbb{H} ^n}{\left| f\left( y \right) \right|^p\left| y \right|_{h}^{\beta}}dy \right) ^{\frac{1}{p}}
		\\
		&=|x|_{h}^{\alpha -Q}\times \left( \omega _Q\times \frac{\left| x \right|_{h}^{Q-\frac{\beta}{p-1}}}{Q-\frac{\beta}{p-1}} \right) ^{\frac{1}{p^{\prime}}}\left\| f \right\| _{L^p(\mathbb{H} ^n,\left| x \right|_{h}^{\beta})}
		\\
		&=\left( \frac{\omega _Q\left( p-1 \right)}{pQ-Q-\beta} \right) ^{\frac{1}{p^{\prime}}}\left\| f \right\| _{L^p(\mathbb{H} ^n,\left| x \right|_{h}^{\beta})}\left| x \right|_{h}^{-\frac{Q}{p}-\frac{\beta}{p}+\alpha}=C_{p,Q,\beta ,f}\left| x \right|_{h}^{-\frac{Q}{p}-\frac{\beta}{p}+\alpha},
	\end{align*}
	where $$C_{p,Q,\beta ,f}=\left( \frac{\omega _Q\left( p-1 \right)}{pQ-Q-\beta} \right) ^{\frac{1}{p^{\prime}}}\left\| f \right\| _{L^p( \mathbb{H} ^n,\left| x \right|_{h}^{\beta})}.$$
	
	\noindent Noticing $|\mathcal{H} _{\alpha}f\left( x \right) |\leqslant C_{p,Q,\beta ,f}\left| x \right|_{h}^{-\frac{Q}{p}-\frac{\beta}{p}+\alpha}$, then we have $\left\{ x:|\mathcal{H} _{\alpha}f\left( x \right) |>\lambda \right\} \subset \{ x:C_{p,Q,\beta ,f}\left| x \right|_{h}^{-\frac{Q}{p}-\frac{\beta}{p}+\alpha}>\lambda \}$. 
	
	\noindent Since 
	$$Q+\gamma >0 \,\,\,\text{and} \,\,\,\frac{\gamma +Q}{q}+\alpha =\frac{\beta +Q}{p},$$ we have
	\begin{align*}
		\left\| \mathcal{H} _{\alpha}f \right\| _{L^{q,\infty}( \mathbb{H} ^n,\left| x \right|_{h}^{\gamma} )}
		&=\underset{\lambda >0}{\mathrm{sup}}\,\,\lambda \left( \int_{\mathbb{H} ^n}{\chi _{\left\{ x:|\mathcal{H} _{\alpha}f\left( x \right) |>\lambda \right\}}}\left( x \right) \left| x \right|_{h}^{\gamma}dx \right) ^{\frac{1}{q}}
		\\
		&\leqslant \underset{\lambda >0}{\mathrm{sup}}\,\,\lambda \left( \int_{\mathbb{H} ^n}{\chi _{\{ x:C_{p,Q,\beta ,f}\left| x \right|_{h}^{-\frac{Q}{p}-\frac{\beta}{p}+\alpha}>\lambda \}}}\left( x \right) \left| x \right|_{h}^{\gamma}dx \right) ^{\frac{1}{q}}
		\\
		&=\underset{\lambda >0}{\mathrm{sup}}\,\,\lambda \left( \int_{B( 0,( \frac{C_{p,Q,\beta ,f}}{\lambda} ) ^{\frac{q}{Q+\gamma}} )}{\left| x \right|_{h}^{\gamma}dx} \right) ^{\frac{1}{q}}=\underset{\lambda >0}{\mathrm{sup}}\,\,\lambda \left( \omega _Q\int_0^{( \frac{C_{p,Q,\beta ,f}}{\lambda} ) ^{\frac{q}{Q+r}}}{r ^{Q-1+\gamma}dr} \right) ^{\frac{1}{q}}
		\\
		&=\underset{\lambda >0}{\mathrm{sup}}\,\,\lambda \left( \omega _Q\times \frac{( C_{p,Q,\beta ,f}/\lambda ) ^q}{Q+\gamma} \right) ^{\frac{1}{q}}=\underset{\lambda >0}{\mathrm{sup}}\,\,C_{p,Q,\beta ,f}\times \left( \frac{\omega _Q}{Q+\gamma} \right) ^{\frac{1}{q}}
		\\
		&=\left( \frac{\omega _Q}{Q+\gamma} \right) ^{\frac{1}{q}}\times \left( \frac{\omega _Q\left( p-1 \right)}{pQ-Q-\beta} \right) ^{\frac{1}{p^{\prime}}}\left\| f \right\| _{L^p( \mathbb{H} ^n,\left| x \right|_{h}^{\beta} )}.
	\end{align*}
	Thus 
	$$\left\| \mathcal{H}_{\alpha} \right\| _{L^p\left( \mathbb{H} ^n,\left| x \right|_{h}^{\beta} \right) \rightarrow L^{q,\infty}\left( \mathbb{H} ^n,\left| x \right|_{h}^{\gamma} \right)}\leqslant \left( \frac{\omega _Q}{Q+\gamma} \right) ^{\frac{1}{q}}\left( \frac{\omega _Q\left( p-1 \right)}{pQ-Q-\beta} \right) ^{\frac{1}{p\prime}}.$$
	On the other hand, let
	$$f_0\left( x \right) =|x|_{h}^{-\frac{\beta}{p-1}}\chi _{\{ x:\left| x \right|_h\leqslant 1 \}}\left( x \right).$$
	Noticing $Q+\beta ( 1-\frac{p}{p-1} ) =Q-\frac{\beta}{p-1}>0$, we have
	\begin{align*}
		\left\| f_0 \right\| _{L^p( \mathbb{H} ^n,\left| x \right|_{h}^{\beta})}
		&=\left( \int_{\mathbb{H} ^n}{|\left| x \right|_{h}^{-\frac{\beta}{p-1}}\chi _{\{x:\left| x \right|_h\leqslant 1\}}}\left. \left( x \right) \right|^p\left| x \right|_{h}^{\beta}dx \right) ^{\frac{1}{p}}
		\\
		&=\left( \int_{|x|_h\leqslant 1}{\left| x \right|_{h}^{-\frac{\beta p}{p-1}}\left| x \right|_{h}^{\beta}}dx \right) ^{\frac{1}{p}}=\left( \omega _Q\int_0^1{r ^{\beta -\frac{p\beta}{p-1}+Q-1}dr} \right) ^{\frac{1}{p}}=\left( \frac{\omega _Q\left( p-1 \right)}{pQ-Q-\beta} \right) ^{\frac{1}{p}}<\infty. 
	\end{align*}
	So we have proved that $f_0\in L^p( \mathbb{H} ^n,\left| x \right|_{h}^{\beta} )$. Then we calculate $\mathcal{H} _{\alpha}\left( f_0 \right) \left( x \right)$.
	\begin{align*}
		\mathcal{H} _{\alpha}\left( f_0 \right) \left( x \right) &=\frac{1}{|x|_{h}^{Q-\alpha}}\int_{|y|_h<|x|_h}{|y|_{h}^{-\frac{\beta}{p-1}}\chi _{\{ y:\left| y \right|_h\leqslant 1 \}}\left( y \right) dy}
		\\
		&=\frac{1}{|x|_{h}^{Q-\alpha}}\left\{ \begin{array}{l}
			\int_{|y|_h<|x|_h}{|y|_{h}^{-\frac{\beta}{p-1}}dy,\quad\left| x \right|_h\leqslant 1}\\
			\int_{|y|_h\leqslant 1}{|y|_{h}^{-\frac{\beta}{p-1}}dy,\quad\,\,\,\,\left| x \right|_h>1}\\
		\end{array} \right. =\frac{\omega _Q}{\left| x \right|_{h}^{Q-\alpha}}\left\{ \begin{array}{l}
			\int_0^{\left| x \right|_h}{r ^{Q-1-\frac{\beta}{p-1}}dr},\quad r \leqslant 1\\
			\int_0^1{r ^{Q-1-\frac{\beta}{p-1}}dr},\quad\,\,\,\, r >1\\
		\end{array} \right. 
		\\
		&=\frac{\omega _Q\left( p-1 \right)}{pQ-Q-\beta}\left\{ \begin{array}{l}
			|x|_{h}^{\alpha -\frac{\beta}{p-1}},\,\,\,\,\,\,\left| x \right|_h\leqslant 1\\
			|x|_{h}^{\alpha -Q},\quad\,\,\,\left| x \right|_h>1\\
		\end{array} \right..
	\end{align*}
	Denote $C_{p,Q,\beta}=\frac{\omega _Q\left( p-1 \right)}{pQ-Q-\beta}$ and 
	$$\left\{ x:|\mathcal{H} _{\alpha}\left( f_0 \right) \left( x \right) |>\lambda \right\} =\{ \left| x \right|_h\leqslant 1:C_{p,Q,\beta}\left| x \right|_{h}^{\alpha -\frac{\beta}{p-1}}>\lambda \} \cup \{ \left| x \right|_h>1:C_{p,Q,\beta}\left| x \right|_h^{\alpha-Q}>\lambda \}.$$
	
	When $0<\lambda <C_{p,Q,\beta}$, noticing $\alpha <\frac{\beta}{p-1}\,\,$ and $\beta <Q\left( p-1 \right)$, we have $\alpha <Q$ and 
	\begin{align*}
		\left\{ x:|\mathcal{H} _{\alpha}\left( f_0 \right) \left( x \right) |>\lambda \right\} 
		&=\left\{ \left| x \right|_h\leqslant 1 \right\} \cup \left\{ \left| x \right|_h>1:\left| x \right|_h\leqslant \left( \frac{C_{p,Q,\beta}}{\lambda} \right) ^{\frac{1}{Q-\alpha}} \right\} 
		\\
		&=\left\{ x:\left| x \right|_h<\left( \frac{C_{p,Q,\beta}}{\lambda} \right) ^{\frac{1}{Q-\alpha}} \right\}.
	\end{align*}
	
	When $\lambda \geqslant C_{p,Q,\beta}$, noticing $\alpha <\frac{\beta}{p-1}\,\,$ and $\beta <Q\left( p-1 \right)$, we have $\alpha <Q$ and 
	$$\left\{ x:|\mathcal{H} _{\alpha}\left( f_0 \right) \left( x \right) |>\lambda \right\} =\left\{ x:\left| x \right|_h<\left( \frac{C_{p,Q,\beta}}{\lambda} \right) ^{\frac{1}{\frac{\beta}{p-1}-\alpha}} \right\}.$$
	Based on the above analysis, we have 
	\begin{align*}
		&\left\| \mathcal{H} _{\alpha}\left( f_0 \right) \right\| _{L^{q,\infty}( \mathbb{H} ^n,\left| x \right|_{h}^{\gamma} )}
		\\
		&=\max \left\{ \underset{0<\lambda <C_{p,Q,\beta}}{\mathrm{sup}}\lambda \left( \int_{\mathbb{H} ^n}{\chi _{\left\{ x:|\mathcal{H} _{\alpha}\left( f_0 \right) \left( x \right) |>\lambda \right\}}\left( x \right) |x|_{h}^{\gamma}dx\,\,} \right) ^{\frac{1}{q}},\underset{C_{p,Q,\beta}\leqslant \lambda}{\mathrm{sup}}\lambda \left( \int_{\mathbb{H} ^n}{\chi _{\left\{ x:|\mathcal{H} _{\alpha}\left( f_0 \right) \left( x \right) |>\lambda \right\}}\left( x \right) |x|_{h}^{\gamma}dx\,\,} \right) ^{\frac{1}{q}} \right\} 
		\\
		&=:\max \left\{ M_1,M_2 \right\} .
	\end{align*}
	Now we first calculate $M_1$. Since 
	$$\left\| f_0 \right\| _{L^p( \mathbb{H} ^n,\left| x \right|_{h}^{\beta} )}=\left( \frac{\omega _Q\left( p-1 \right)}{pQ-Q-\beta} \right) ^{\frac{1}{p}},\quad\gamma >-Q,	$$
	and
	$$1-\frac{Q+\gamma}{\left( Q-\alpha \right) q}=1-\frac{1}{Q-\alpha}( \frac{\beta +Q}{p}-\alpha ) =\frac{Q\left( p-1 \right) -\beta}{p\left( Q-\alpha \right)}>0,$$
	we have 
	\begin{align*}
		M_1&=\underset{0<\lambda <C_{p,Q,\beta}}{\mathrm{sup}}\lambda \left( \int_{\mathbb{H} ^n}{\chi _{\left\{ x:|\mathcal{H} _{\alpha}\left( f_0 \right) \left( x \right) |>\lambda \right\}}\left( x \right) |x|_{h}^{\gamma}dx\,\,} \right) ^{\frac{1}{q}}=\underset{0<\lambda <C_{p,Q,\beta}}{\mathrm{sup}}\lambda \left( \int_{\left| x \right|_h<( \frac{C_{p,Q,\beta}}{\lambda} ) ^{\frac{1}{Q-\alpha}}}{|x|_{h}^{\gamma}dx\,\,} \right) ^{\frac{1}{q}}
		\\
		&=\underset{0<\lambda <C_{p,Q,\beta}}{\mathrm{sup}}\left( \frac{\omega _Q}{Q+\gamma} \right) ^{\frac{1}{q}}\left( C_{p,Q,\beta} \right) ^{\frac{Q+\gamma}{\left( Q-\alpha \right) q}}\lambda ^{1-\frac{Q+\gamma}{\left( Q-\alpha \right) q}}
		\\
		&=\left( \frac{\omega _Q}{Q+\gamma} \right) ^{\frac{1}{q}}C_{p,Q,\beta}=\left( \frac{\omega _Q}{Q+\gamma} \right) ^{\frac{1}{q}}\times \left( \frac{\omega _Q\left( p-1 \right)}{pQ-Q-\beta} \right) ^{\frac{1}{p}+\frac{1}{p^{\prime}}}
		\\
		&=\left( \frac{\omega _Q}{Q+\gamma} \right) ^{\frac{1}{q}}\left( \frac{\omega _Q\left( p-1 \right)}{pQ-Q-\beta} \right) ^{\frac{1}{p^{\prime}}}\left\| f_0 \right\| _{L^p( \mathbb{H} ^n,\left| x \right|_{h}^{\beta} )}.
	\end{align*}
	Then we calculate $M_2$, noticing $\left\| f_0 \right\| _{L^p( \mathbb{H} ^n,\left| x \right|_{h}^{\beta} )}=\left( \frac{\omega _Q\left( p-1 \right)}{pQ-Q-\beta} \right) ^{\frac{1}{p}}$, $\gamma >-Q$\, and 
	$$1-\frac{Q+\gamma}{( \frac{\beta}{p-1}-\alpha) q}=1-\frac{1}{\frac{\beta}{p-1}-\alpha}( \frac{\beta +Q}{p}-\alpha ) =\frac{\beta -Q\left( p-1 \right)}{p(\frac{\beta}{p-1}-\alpha )(p-1)}<0,$$
	we have
	\begin{align*}
		M_2&=\underset{C_{p,Q,\beta}\leqslant  \lambda}{\mathrm{sup}}\lambda \left( \int_{\mathbb{H} ^n}{\chi _{\left\{ x:|\mathcal{H} _{\alpha}\left( f_0 \right) \left( x \right) |>\lambda \right\}}\left( x \right) |x|_{h}^{\gamma}dx\,\,} \right) ^{\frac{1}{q}}=\underset{C_{p,Q,\beta}\geqslant \lambda}{\mathrm{sup}}\lambda \left( \int_{\left| x \right|_h<( \frac{C_{p,Q,\beta}}{\lambda}) ^{\frac{1}{\frac{\beta}{p-1}-\alpha}}}{|x|_{h}^{\gamma}dx\,\,} \right) ^{\frac{1}{q}}
		\\
		&=\underset{C_{p,Q,\beta}\leqslant  \lambda}{\mathrm{sup}}\left( \frac{\omega _Q}{Q+r} \right) ^{\frac{1}{q}}\left( C_{p,Q,\beta} \right) ^{\frac{Q+\gamma}{( \frac{\beta}{p-1}-\alpha )}}\lambda ^{1-\frac{Q+\gamma}{( \frac{\beta}{p-1}-\alpha )}}
		\\
		&=\left( \frac{\omega _Q}{Q+\gamma} \right) ^{\frac{1}{q}}C_{p,Q,\beta}=\left( \frac{\omega _Q}{Q+r} \right) ^{\frac{1}{q}}\times \left( \frac{\omega _Q\left( p-1 \right)}{pQ-Q-\beta} \right) ^{\frac{1}{p}+\frac{1}{p^{\prime}}}
		\\
		&=\left( \frac{\omega _Q}{Q+\gamma} \right) ^{\frac{1}{q}}\left( \frac{\omega _Q\left( p-1 \right)}{pQ-Q-\beta} \right) ^{\frac{1}{p^{\prime}}}\left\| f_0 \right\| _{L^p( \mathbb{H} ^n,\left| x \right|_{h}^{\beta})}.
	\end{align*}
	Its easy to see that $M_1=M_2$, and then
	$$\left\| \mathcal{H} _{\alpha} \right\| _{L^p\left( \mathbb{H} ^n,\left| x \right|_{h}^{\beta} \right) \rightarrow \,\,\left\| \mathcal{H} _{\alpha} \right\| _{L^{p,\infty}\left( \mathbb{H} ^n,\left| x \right|_{h}^{\gamma} \right)}}=\left( \frac{\omega _Q}{Q+\gamma} \right) ^{\frac{1}{q}}\left( \frac{\omega _Q\left( p-1 \right)}{pQ-Q-\beta} \right) ^{\frac{1}{p^{\prime}}}.$$
	This finishes the proof of Theorem \ref{thm2.1}.
\end{proof}\renewcommand{\qedsymbol}{}
\begin{proof}[Proof of Theorem \ref{thm2.2}.]\renewcommand{\qedsymbol}{}
	It is easy to see that
	$$\left| \mathcal{H} _{\alpha}f\left( x \right) \right|=\left| \frac{1}{|x|_{h}^{Q-a}}\int_{|y|_h<|x|_h}{f\left( y \right) dy} \right|\leqslant \left| \frac{1}{|x|_{h}^{Q-a}}\int_{\mathbb{H} ^n}{f\left( y \right)}dy \right|=|x|_{h}^{\alpha -Q}\left\| f \right\| _{L^1\left( \mathbb{H} ^n \right)}.$$
	Notice $\left| \mathcal{H} _{\alpha}f\left( x \right) \right|\leqslant |x|_{h}^{\alpha -Q}\left\| f \right\| _{L^1\left( \mathbb{H} ^n \right)}$, and we have $\left\{ x:\left| \mathcal{H} _{\alpha}f\left( x \right) \right|>\lambda \right\} \subset \{ x:|x|_{h}^{\alpha -Q}\left\| f \right\| _{L^1\left( \mathbb{H} ^n \right)}>\lambda \}$. Since $Q-\alpha >0$ and $Q+\gamma >0$, we have
	\begin{align*}
		&\left\| \mathcal{H} _{\alpha}f \right\| _{L^{\left( Q+\gamma \right) /\left( Q-\alpha \right) ,\infty}\left( \mathbb{H} ^n,\left| x \right|_{h}^{\gamma} \right)}=\underset{\lambda >0}{\mathrm{sup}}\,\,\lambda \left( \int_{\mathbb{H} ^n}{\chi _{\left\{ x:|\mathcal{H} _{\alpha}f\left( x \right) |>\lambda \right\}}\left( x \right) |x|_{h}^{\gamma}dx\,\,} \right) ^{\frac{Q-\alpha}{Q+\gamma}}
		\\
		&\leqslant \underset{\lambda >0}{\mathrm{sup}}\,\,\lambda \left( \int_{\mathbb{H} ^n}{\chi _{\{ x:|x|_{h}^{\alpha -Q}\left\| f \right\| _{L^1\left( \mathbb{H} ^n \right)}>\lambda\}}\left( x \right) |x|_{h}^{\gamma}dx\,\,} \right) ^{\frac{Q-\alpha}{Q+\gamma}}=\underset{\lambda >0}{\mathrm{sup}}\,\,\lambda \left( \int_{\left| x \right|_h<( \left\| f \right\| _{L^1\left( \mathbb{H} ^n \right)}/\lambda) ^{\frac{1}{Q-\alpha}}}{|x|_{h}^{\gamma}dx\,\,} \right) ^{\frac{Q-\alpha}{Q+\gamma}}
		\\
		&=\underset{\lambda >0}{\mathrm{sup}}\,\,\lambda \left( \omega _Q\int_0^{( \left\| f \right\| _{L^1\left( \mathbb{H} ^n \right)}/\lambda ) ^{\frac{1}{Q-\alpha}}}{r^{Q-1+\gamma}dr} \right) ^{\frac{Q-\alpha}{Q+\gamma}}=\left( \frac{\omega _Q}{Q+\gamma} \right) ^{\frac{Q-\alpha}{Q+\gamma}}\left\| f \right\| _{L^1\left( \mathbb{H} ^n \right)}.
	\end{align*}
	Thus 
	$$\left\| \mathcal{H} _{\alpha}f\left( x \right) \right\| _{L^{\left( Q+\gamma \right) /\left( Q-\alpha \right) ,\infty}\left( \mathbb{H} ^n,\left| x \right|_{h}^{\gamma} \right)}\leqslant \left( \frac{\omega _Q}{Q+\gamma} \right) ^{\frac{Q-\alpha}{Q+\gamma}}\left\| f \right\| _{L^1\left( \mathbb{H} ^n \right)}.$$
	On the other hand, let $f_0\left( x \right) =\chi _{\{ x:\left| x \right|_h\leqslant 1 \}}\left( x \right)$.
	Then we have
	$$\left\| f_0 \right\| _{L^1\left( \mathbb{H} ^n \right)}=\int_{\mathbb{H} ^n}{\chi _{\{x:\left| x \right|_h\leqslant 1\}}\left( x \right) dx=\frac{\omega _Q}{Q}}<\infty ,$$
	and so $f_0\in L^1\left( \mathbb{H} ^n \right)$ and
	\begin{align*}
		\mathcal{H} _{\alpha}\left( f_0 \right) \left( x \right) &=\frac{1}{|x|_{h}^{Q-a}}\int_{|y|_h<|x|_h}{\chi _{\left\{ y:|y|_h\leqslant 1 \right\}}\left( y \right) dy}
		\\
		&=\frac{1}{|x|_{h}^{Q-a}}\left\{ \begin{array}{l}
			\int_{|y|_h<|x|_h}{dy},\quad |x|_h\leqslant 1\\
			\int_{|y|_h\leqslant 1}{dy},\quad \,\,\,\, |x|_h>1\\
		\end{array} \right.=\frac{\omega _Q}{Q}\left\{ \begin{array}{l}
			|x|_{h}^{\alpha},\quad \,\,\,\,\,\, |x|_h\leqslant 1\\
			|x|_{h}^{\alpha -Q},\quad|x|_h>1\\
		\end{array} \right..
	\end{align*}
	Denote $C_Q=\frac{\omega _Q}{Q}$ and $$\left\{ x:|\mathcal{H} _{\alpha}\left( f_0 \right) \left( x \right) |>\lambda \right\} =\left\{ |x|_h\leqslant 1:|x|_{h}^{\alpha}C_Q>\lambda \right\} \cup \{ |x|_h>1:|x|_{h}^{\alpha -Q}C_Q>\lambda \}.$$
	
	When $\lambda \geqslant C_Q$, noticing $0<\alpha <Q$, we have $\left\{ x:|\mathcal{H} _{\alpha}\left( f_0 \right) \left( x \right) |>\lambda \right\} =\varnothing$.
	
	When $0<\lambda <C_Q$, noticing $0<\alpha <Q$, we have 
	$$\left\{ x:|\mathcal{H} _{\alpha}\left( f_0 \right) \left( x \right) |>\lambda \right\}=\left\{ x:\left( \frac{\lambda}{C_Q} \right) ^{\frac{1}{\alpha}}<\left| x \right|_h<\left( \frac{C_Q}{\lambda} \right) ^{\frac{1}{Q-\alpha}} \right\}.$$
	We have
	\begin{align*}
		&\left\| \mathcal{H} _{\alpha}\left( f_0 \right) \left( x \right) \right\| _{L^{\left( Q+\gamma \right) /\left( Q-\alpha \right) ,\infty}\left( \mathbb{H} ^n,\left| x \right|_{h}^{\gamma} \right)}
		\\
		&=\max \left\{ \underset{0<\lambda <C_Q}{\mathrm{sup}}\lambda \left( \int_{\mathbb{H} ^n}{\chi _{\left\{ x:|\mathcal{H} _{\alpha}f_0\left( x \right) |>\lambda \right\}}\left( x \right) |x|_{h}^{\gamma}dx\,\,} \right) ^{\frac{Q-\alpha}{Q+\gamma}},\underset{\lambda \geqslant C_Q}{\mathrm{sup}}\lambda \left( \int_{\mathbb{H} ^n}{\chi _{\left\{ x:|\mathcal{H} _{\alpha}f_0\left( x \right) |>\lambda \right\}}\left( x \right) |x|_{h}^{\gamma}dx\,\,} \right) ^{\frac{Q-\alpha}{Q+\gamma}} \right\} 
		\\
		&=:\max \left\{ M_3,M_4 \right\} .
	\end{align*}
	When $\lambda \geqslant C_Q$, then $\left\{ x:|\mathcal{H} _{\alpha}f_0\left( x \right) |>\lambda \right\} =\varnothing$, and we have $M_4=0$. Then, we only need to calculate $M_3$.
	In addition, noticing 
	$$Q+\beta >0,\quad0<\alpha <Q,\quad\left\| f_0 \right\| _{L^1\left( \mathbb{H} ^n \right)}=\frac{\omega _Q}{Q},$$
	we have
	\begin{align*}
		M_3&=\underset{0<\lambda <C_Q}{\mathrm{sup}}\lambda \left( \int_{\mathbb{H} ^n}{\chi _{\left\{ x:|\mathcal{H} _{\alpha}f_0\left( x \right) |>\lambda \right\}}\left( x \right) |x|_{h}^{\gamma}dx\,\,} \right) ^{\frac{Q-\alpha}{Q+\gamma}}
		\\
		&=\underset{0<\lambda <C_Q}{\mathrm{sup}}\lambda \left( \int_{(\frac{\lambda}{C_Q})^{\frac{1}{\alpha}}<\left| x \right|_h<(\frac{C_Q}{\lambda})^{\frac{1}{Q-\alpha}}}{|x|_{h}^{\gamma}dx\,\,} \right) ^{\frac{Q-\alpha}{Q+\gamma}}=\underset{0<\lambda <C_Q}{\mathrm{sup}}\lambda \left( \omega _Q\int_{(\frac{\lambda}{C_Q})^{\frac{1}{\alpha}}}^{(\frac{C_Q}{\lambda})^{\frac{1}{Q-\alpha}}}{r^{Q-1+\gamma}dr} \right) ^{\frac{Q-\alpha}{Q+\gamma}}
		\\
		&=\underset{0<\lambda <C_Q}{\mathrm{sup}}\left( \frac{\omega _Q}{Q+\gamma} \right) ^{\frac{Q-\alpha}{Q+\gamma}}\left( C_{Q}^{\frac{Q+\gamma}{Q-\alpha}}-\frac{\lambda ^{\frac{Q+\gamma}{\alpha}+\frac{Q+\gamma}{Q-\alpha}}}{C_{Q}^{\frac{Q+\gamma}{\alpha}}} \right) ^{\frac{Q-\alpha}{Q+\gamma}}
		\\
		&=\left( \frac{\omega _Q}{Q+\gamma} \right) ^{\frac{Q-\alpha}{Q+\gamma}}C_Q=\left( \frac{\omega _Q}{Q+\gamma} \right) ^{\frac{Q-\alpha}{Q+\gamma}}\left\| f_0 \right\| _{L^1\left( \mathbb{H} ^n \right)}.
	\end{align*}
	Then
	$$\left\| \mathcal{H} _{\alpha} \right\| _{L^1\left( \mathbb{H} ^n \right) \rightarrow L^{\left( Q+\gamma \right) /\left( Q-\alpha \right) ,\infty}\left( \mathbb{H} ^n,\left| x \right|_{h}^{\gamma} \right)}=\left( \frac{\omega _Q}{Q+\gamma} \right) ^{\frac{Q-\alpha}{Q+\gamma}}.$$
	This finishes the proof of Theorem \ref{thm2.2}. Notices that Theorem 2.2 no longer holds when $\alpha=0$.
\end{proof}\renewcommand{\qedsymbol}{}

\section{Sharp weighted $\boldsymbol{L}^{\boldsymbol{p}}$ estimate for the integral operator with a kernel}
In this section, we will study the $m$-linear $n$-dimensional integral operator with a kernel on the Heisenberg group. Let  $K:\mathbb{H} ^n\times \cdots \times \mathbb{H} ^n\rightarrow \left( 0,\infty \right)$ be a measurable radial kernel such that $K\left( y_1,...,y_m \right) =K(\left| y_1 \right|_h,...,\left| y_m \right|_h) $, satisfying 
\begin{align}\label{3.1}
	C^h=\int_{\mathbb{H} ^n}{\cdots \int_{\mathbb{H} ^n}{K\left( y_1,...,y_m \right) \prod_{i=1}^m{\left| y_i \right|_{h}^{-\frac{\alpha _j}{q}-\frac{Q}{q_j}}}}}dy_1\cdots dy_m<\infty,
\end{align}
where $\alpha_j$, $q_j$, $q$, and $Q$ are the pre-defined indicator and some fixed indices, $j=1,2,...,m$. The $m$-linear $n$-dimensional integral operator with a kernel is defined by
\begin{align}\label{3.2}
	\mathcal{H} ^h\left( f_1,...,f_m \right) \left( x \right) =\int_{\mathbb{H} ^n}{\cdots}\int_{\mathbb{H} ^n}{K\left( y_1,...,y_m \right) f_1(\delta _{\left| x \right|_h}y_1)\cdots}f_m(\delta _{\left| x \right|_h}y_m)dy_1\cdots dy_m,
\end{align}
where $x\in \mathbb{R} ^n\backslash \left\{ 0 \right\}$ and $f_j$ is a measurable radial function on $\mathbb{H} ^n$ with $j=1,2,...,m$. Note that $\mathcal{H}^h$ is in fact an integral operator having a homogeneous radial $K$ of degree $-mn$.

In this paper, we will give the weighted $L^p$ estimate for the $m$-linear $n$-dimensional integral operator with a kernel on the Heisenberg group.
\begin{thm}\label{thm3.1}
	Let $m\in \mathbb{N}$, $1<q<\infty$, $\frac{1}{q}=\frac{1}{q_1}+\cdots +\frac{1}{q_m}$, $\alpha =\alpha _1+\cdots +\alpha _m$, $1<q_j<\infty$ with $j=1,...,m$, and $f_j$ be a radial function in $L^{q_{\boldsymbol{j}}}(\mathbb{H} ^n,\left| x \right|_{h}^{\frac{q_{\boldsymbol{j}}\alpha _{\boldsymbol{j}}}{q}})$. Assume that the kernel $K$ is the constant defined by (\ref{3.1}).
	
	\noindent Then
	\begin{equation}\label{3.3}
		\begin{aligned}
			\left\| \mathcal{H}^h\right\| _{L^{q_1}( \mathbb{H} ^n,\left| x \right|_{h}^{\frac{q_1\alpha _1}{q}} ) \times \cdot \cdot \cdot \times L^{q_m}( \mathbb{H} ^n,\left| x \right|_{h}^{\frac{q_m\alpha _m}{q}} ) \rightarrow L^q( \mathbb{H} ^n,\left| x \right|_{h}^{\alpha} )}=C^h.
		\end{aligned}
	\end{equation}
\end{thm}
\begin{proof}[Proof.]\renewcommand{\qedsymbol}{}
	Consider that
	$$g_j\left( x \right) =\frac{1}{\omega _Q}\int_{\left| \xi _j \right|_{h=1}}{f_j(\delta _{\left| x \right|_h}\xi _j)}d\xi _j,\quad x\in \mathbb{H} ^n,\quad j=1,...,m.$$
	Obviously, $g_j$ satisfies $g_j\left( x \right) =g_j\left( \left| x \right|_h \right)$, and $\mathcal{H} ^h\left( f_1,...,f_m \right) \left( x \right)$ is equal to
	\begin{align*}
		&\mathcal{H} ^h\left( g_1,...,g_m \right) \left( x \right) 
		\\
		&=\int_{\mathbb{H} ^{mn}}{K\left( y_1,...,y_m \right) g_1(\delta _{| x |_h}y_1)\cdots g_m(\delta _{| x |_h}y_m)dy_1\cdots dy_m}
		\\
		&=\int_{\mathbb{H} ^{mn}}{K\left( y_1,...,y_m \right)}\prod_{j=1}^m{\left( \frac{1}{\omega _Q}\int_{| \xi _j |_h=1}{f_j(\delta _{\left| x \right|_h| y_j |_h}\xi _j)}d\xi _j \right)}dy_1\cdots dy_m
		\\
		&=\frac{1}{\omega _{Q}^{m}}\int_{\mathbb{H} ^{mn}}{K\left( y_1,...,y_m \right)}\prod_{j=1}^m{\left( \int_{| \xi _j |_h=1}{f_j(\delta _{\left| x \right|_h}\left( \delta _{| y_j |_h}\xi _j \right) )}d\xi _j \right)}dy_1\cdots dy_m
		\\
		&=\frac{1}{\omega _{Q}^{m}}\int_{\mathbb{H} ^{mn}}{K\left( y_1,...,y_m \right)}\prod_{j=1}^m{\left( \int_{| z_j |_h=| y_j |_h}{f_j(\delta _{\left| x \right|_h}z_j)| y_j |_{h}^{-Q}}dz_j \right)}dy_1\cdots dy_m
		\\
		&=\frac{1}{\omega _{Q}^{m}}\int_{\mathbb{H} ^{mn}}{\int_{\left| y_1 \right|_h=\left| z_1 \right|_h}{\cdots \int_{\left| y_m \right|_h=\left| z_m \right|_h}{K(\left| y_1 \right|_h,...,\left| y_m \right|_h)}}}\prod_{j=1}^m{f_j(\delta _{\left| x \right|_h}z_j)| y_j |_{h}^{-Q}}dy_1\cdots dy_mdz_m\cdots dz_1
		\\
		&=\frac{1}{\omega _{Q}^{m}}\int_{\mathbb{H} ^{mn}}{\int_{|t_1|_h=1}{\cdots \int_{|t_m|_h=1}{K(|z_1|_h,...,|z_m|_h)}}}\prod_{j=1}^m{f_j(\delta _{\left| x \right|_h}z_j)}dt_1\cdots dt_mdz_m\cdots dz_1
		\\
		&=\int_{\mathbb{H} ^{mn}}{K(z_1,...,z_m)}f_1(\delta _{\left| x \right|_h}z_1)\cdots f_m(\delta _{\left| x \right|_h}z_m)dz_1\cdots dz_m=\mathcal{H} ^h(f_1,...,f_m)(x).
	\end{align*}
	In the fourth to fifth lines, we let $z_j=\delta _{| y_j |_h}\xi _j$. From the fifth to sixth lines, we perform an integral permutation. In the sixth to seventh lines, we set $y_j=\delta _{| z_j |_h}t_j$. On the other hand, by applying Hölder's inequality, we conclude that
	\begin{align*}
		\left\| g_j \right\| _{L^{q_j}( \mathbb{H} ^n,\left| x \right|_{h}^{\frac{q_j\alpha _j}{q}} )}
		&=\left( \int_{\mathbb{H} ^n}{\left| \frac{1}{\omega _Q}\int_{| \xi _j |_h=1}{f_j(\delta _{\left| x \right|_h}\xi _j)}d\xi _j \right|^{q_j}\left| x \right|_{h}^{\frac{q_j\alpha _j}{q}}dx} \right) ^{\frac{1}{q_j}}
		\\
		&=\frac{1}{\omega _Q}\left( \int_{\mathbb{H} ^n}{\left| \int_{| \xi _j |_h=1}{f_j(\delta _{\left| x \right|_h}\xi _j)}d\xi _j \right|^{q_j}\left| x \right|_{h}^{\frac{q_j\alpha _j}{q}}dx} \right) ^{\frac{1}{q_j}}
		\\
		&\leqslant \frac{1}{\omega _Q}\left( \int_{\mathbb{H} ^n}{\int_{| \xi _j |_h=1}{|f_j(\delta _{\left| x \right|_h}\xi _j)|^{q_j}}d\xi _j}\left( \int_{| \xi _j |_h=1}{dx} \right) ^{q_j-1}| x |_{h}^{\frac{q_j\alpha _j}{q}}dx \right) ^{\frac{1}{q_j}}
		\\
		&={\omega _Q}^{-\frac{1}{q_j}}\left( \int_{\mathbb{H} ^n}{\int_{| \xi _j |_h=1}{|f_j(\delta _{\left| x \right|_h}\xi _j)|^{q_j}}d\xi _j}\left| x \right|_{h}^{\frac{q_j\alpha _j}{q}}dx \right) ^{\frac{1}{q_j}}
		\\
		&={\omega _Q}^{-\frac{1}{q_j}}\left( \int_{\mathbb{H} ^n}{\int_{| z_j |_h=\left| x \right|_h}{| f_j( z_j ) |^{q_j}}\left| x \right|_{p}^{-Q}dz_j}\left| x \right|_{h}^{\frac{q_j\alpha _j}{q}}dx \right) ^{\frac{1}{q_j}}
		\\
		&={\omega _Q}^{-\frac{1}{q_j}}\left( \int_{\mathbb{H} ^n}{\int_{\left| x \right|_h=| z_j |_h}{\left| x \right|_{h}^{-n}\left| x \right|_{h}^{\frac{q_j\alpha _j}{q}}dx| f_j( z_j ) |^{q_j}}dz_j} \right) ^{\frac{1}{q_j}}
		\\
		&={\omega _Q}^{-\frac{1}{q_j}}\left( \int_{\mathbb{H} ^n}{\int_{| t_j |_h=1}{| z_j |_{h}^{-Q}| t_j |_{h}^{-Q}| z_j |_{h}^{\frac{q_j\alpha _j}{q}}| t_j |_{h}^{\frac{q_j\alpha _j}{q}}}| f_j( z_j ) |^{q_j}| z_j |_{h}^{Q}dt_jdz_j} \right) ^{\frac{1}{q_j}}
		\\
		&=\left( \int_{\mathbb{H} ^n}{| f_j( z_j ) |^{q_j}| z_j |_{h}^{\frac{q_j\alpha _j}{q}}}dz_j \right) ^{\frac{1}{q_j}}=\left\| f_j \right\| _{L^{q_j}(\mathbb{H} ^n,\left| x \right|_{h}^{\frac{q_j\alpha _j}{q}})}.
	\end{align*}
	From the second to third lines, we apply Hölder's inequality. In the fourth to fifth lines, we let $z_j=\delta _{\left| x \right|_h}\xi _j$. From the fifth to sixth lines, we perform an integral permutation. In the sixth to seventh lines, we set $x=\delta _{| z_j |_h}t_j$.
	Therefore we have  
	$$\frac{\left\| \mathcal{H} ^h\left( f_1,...,f_m \right) \right\| _{L^q(\mathbb{H} ^n,\left| x \right|_{h}^{\alpha})}}{\prod_{j=1}^m{\left\| f_j \right\| _{L^{q_{\boldsymbol{j}}}(\mathbb{H} ^n,\left| x \right|_{h}^{\frac{q_{\boldsymbol{j}}\alpha _{\boldsymbol{j}}}{q}})}}}\leqslant \frac{\left\| \mathcal{H} ^h\left( g_1,...,g_m \right) \right\| _{L^q(\mathbb{H} ^n,\left| x \right|_{h}^{\alpha})}}{\prod_{j=1}^m{\left\| g_j \right\| _{L^{q_{\boldsymbol{j}}}(\mathbb{H} ^n,\left| x \right|_{h}^{\frac{q_{\boldsymbol{j}}\alpha _{\boldsymbol{j}}}{q}})}}},$$
	which implies that the operator $\mathcal{H} ^h$ and its restriction to the function $g$ satisfying $g_j\left( x \right) =g_j(\left| x \right|_h)$ have the same operator norm in $L^q(\mathbb{H} ^n,\left| x \right|_{h}^{\alpha})$. So without loss of generality, we assume that $f_j\in L^{q_{\boldsymbol{j}}}(\mathbb{H} ^n,\left| x \right|_{h}^{\frac{q_{\boldsymbol{j}}\alpha _{\boldsymbol{j}}}{q}})$ with $j=1,2,...,m$ satisfies that $f_j\left( x \right) =f_j( \left| x \right|_h)$ in the rest of the proof.
	Let $q^{\prime}$ is the conjugate number of $q$ and $g\in L^{q^{\prime}}(\mathbb{H} ^n,\left| x \right|_{h}^{\alpha})$. Using duality and Holder's inequality, and making a change of variables, we obtain the following sequence of inequalities:
	\begin{align*}
		&\left| \left< \mathcal{H} ^h\left( f_1,... ,f_m \right) ,g \right> \right|\,\,      
		\\
		&\leqslant \int_{\mathbb{H} ^{mn}}{\left| K\left( y_1,...,y_m \right) \right|}\int_{\mathbb{H} ^n}{\left| g\left( x \right) \right|}\left| f_1( \delta _{|x|_h}y_1 ) \right|\cdots \left| f_m( \delta _{|x|_h}y_m ) \right||x|_{h}^{\alpha}dxdy_1\cdots dy_m
		\\
		&=\int_{\mathbb{H} ^{mn}}{\left| K\left( y_1,... ,y_m \right) \right|}\int_{\mathbb{H} ^n}{\left| g\left( x \right) \right|}|x|_{h}^{\frac{\alpha}{q^{\prime}}}| f_1( \delta _{|x|_h}y_1) |\cdots | f_m( \delta _{|x|_h}y_m ) ||x|_{h}^{\frac{\alpha}{q}}dxdy_1\cdots dy_m
		\\
		&\leqslant \int_{\mathbb{H} ^{mn}}{\left| K\left( y_1,... ,y_m \right) \right|}\left( \int_{\mathbb{H} ^n}{\left| g\left( x \right) \right|^{q^{\prime}}}|x|_{h}^{\alpha}dx \right) ^{\frac{1}{q^{\prime}}}\left( \int_{\mathbb{H} ^n}{\left( | f_1( \delta _{|x|_h}y_1) |\cdots | f_m( \delta _{|x|_h}y_m)||x|_{h}^{\frac{\alpha}{q}} \right) ^qdx} \right) ^{\frac{1}{q}}dy_1\cdots dy_m
		\\
		&\leqslant \left\| g \right\| _{L^{q^{\prime}}\left( \mathbb{H} ^n,\left| x \right|_{h}^{\alpha} \right)}\int_{\mathbb{H} ^{mn}}{\left| K\left( y_1,...,y_m \right) \right|}\prod_{j=1}^m{\left( \int_{\mathbb{H} ^n}{\left| f_j( \delta _{|y_j|_h}x ) \right|^{q_j}|x|_{h}^{\frac{q_j\alpha _j}{q}}dx} \right) ^{\frac{1}{q_j}}}dy_1\cdots dy_m
		\\
		&=\left\| g \right\| _{L^{q^{\prime}}\left( \mathbb{H} ^n,\left| x \right|_{h}^{\alpha} \right)}\int_{\mathbb{H} ^{mn}}{\left| K\left( y_1,...,y_m \right) \right|}\prod_{j=1}^m{|y|_{h}^{-\frac{\alpha _j}{q}-\frac{Q}{q_j}}\prod_{j=1}^m{\left( \int_{\mathbb{H} ^n}{| f_j( z_j ) |^{q_j}|z_j|_{h}^{\frac{q_j\alpha _j}{q}}dx} \right)}^{\frac{1}{q_j}}}dy_1\cdots dy_m
		\\
		&=\int_{\mathbb{H} ^{mn}}{K\left( y_1,...,y_m \right)}\prod_{j=1}^m{|y|_{h}^{-\frac{\alpha _j}{q}-\frac{Q}{q_j}}}dy_1\cdot \cdot \cdot dy_m\left\| g \right\| _{L^{q^{\prime}}\left( \mathbb{H} ^n,\left| x \right|_{h}^{\alpha} \right)}\prod_{j=1}^m{\left\| f_j \right\| _{L^{q_j}( \mathbb{H} ^n,\left| x \right|_{h}^{\frac{q_j\alpha _j}{q}})}}
		\\
		&=C^h\left\| g \right\| _{L^{q^{\prime}}\left( \mathbb{H} ^n,\left| x \right|_{h}^{\alpha} \right)}\prod_{j=1}^m{\left\| f_j \right\| _{L^{q_j}( \mathbb{H} ^n,\left| x \right|_{h}^{\frac{q_j\alpha _j}{q}} )}}.
	\end{align*}
	This proves the first part of our theorem.
	
	For the second part, we will show that if the kernel $K$ is nonnegative, then the operator norm $\| \mathcal{H}^h \|$ of $\mathcal{H}^h$ is equal to $C^h$. For a positive number $\varepsilon$ and $i=1,2,...,m$, we define the sequences of functions $g_{\varepsilon}$ and $f_{j,\varepsilon}$ by
	$$g_{\varepsilon}\left( x \right) =|x|_{h}^{-\frac{Q+\alpha}{q^{\prime}}+\frac{\varepsilon}{q^{\prime}}}\chi _{B\left( 0,1 \right)}\left( x \right) ,\quad f_{j,\varepsilon}=|x|_{h}^{-\frac{\alpha _j}{q}-\frac{Q}{q_j}+\frac{\varepsilon}{q_j}}\chi _{B\left( 0,1 \right)}\left( x \right).$$
	By a simple computation,  we have
	\begin{align*}
		{\left\| g_{\varepsilon} \right\| ^{q^{\prime}}}_{L^{q^{\prime}}\left( \mathbb{H} ^n,\left| x \right|_{h}^{\alpha} \right)}
		&=\int_{\mathbb{H} ^n}{|x|_{h}^{-Q+\varepsilon}dx}=\left( \left( \int_{\mathbb{H} ^n}{\left( |x|_{h}^{-\frac{\alpha _j}{q}-\frac{Q}{q_j}+\frac{\varepsilon}{q_j}} \right) ^{q_j}\left| x \right|_{h}^{\frac{q_j\alpha _j}{q}}dx} \right) ^{\frac{1}{q_j}} \right) ^{q_j}
		\\
		&={\left\| f_{j,\varepsilon} \right\| ^{q^j}}_{L^{q^j}( \mathbb{H} ^n,\left| x \right|_{h}^{\frac{q_j\alpha _j}{q}})}=\left\| g_{\varepsilon} \right\| _{L^{q^{\prime}}( \mathbb{H} ^n,\left| x \right|_{h}^{\alpha} )}\prod_{j=1}^m{\left\| f_{j,\varepsilon} \right\| _{L^{q_j}( \mathbb{H} ^n,\left| x \right|_{h}^{\frac{q_j\alpha _j}{q}})}}=\frac{\omega _Q}{\varepsilon}.
	\end{align*}
	Therefore, we have
	\begin{align*}
		&\left| \left< \mathcal{H} ^h\left( f_{1,\varepsilon},...,f_{m,\varepsilon} \right) ,g_{\varepsilon} \right> \right|
		\\
		&=\int_{B\left( 0,1 \right)}{|x|_{h}^{-\frac{Q+\alpha}{q^{\prime}}+\frac{\varepsilon}{q^{\prime}}}}|x|_{h}^{\alpha}\int_{\mathbb{H} ^{mn}}{\left| K\left( y_1,...,y_m \right) \right|}\prod_{j=1}^m{f_{j,\varepsilon}(\delta _{|x|_h}y_j)dy_1\cdots dy_mdx}
		\\
		&=\int_{B\left( 0,1 \right)}{|x|_{h}^{-\frac{Q+\alpha}{q^{\prime}}+\frac{\varepsilon}{q^{\prime}}}}\left| x \right|_{h}^{\alpha}\int_{B(0,\frac{1}{\left| x \right|_h})}{\cdots}\int_{B(0,\frac{1}{\left| x \right|_h})}{K\left( y_1,...,y_m \right)}\prod_{j=1}^m{\left( \left| x \right|_h| y_j |_h \right) ^{-\frac{\alpha _j}{q}-\frac{Q}{q_j}+\frac{\varepsilon}{q_j}}}dy_1\cdots dy_mdx
		\\
		&=\int_{B\left( 0,1 \right)}{|x|_{h}^{-Q+\varepsilon}}\int_{B(0,\frac{1}{\left| x \right|_h})}{\cdots}\int_{B(0,\frac{1}{\left| x \right|_h})}{K\left( y_1,...,y_m \right)}\prod_{j=1}^m{| y_j |_{h}^{-\frac{\alpha _j}{q}-\frac{Q}{q_j}+\frac{\varepsilon}{q_j}}}dy_1\cdots dy_mdx
		\\
		&=\omega _Q\int_0^1{r^{\varepsilon -1}}\int_{B( 0,\frac{1}{r} )}{\cdots \int_{B( 0,\frac{1}{r} )}{K\left( y_1,...,y_m \right)}}\prod_{j=1}^m{| y_j |_{h}^{-\frac{\alpha _j}{q}-\frac{Q}{q_j}+\frac{\varepsilon}{q_j}}}dy_1\cdots dy_mdr
		\\
		&=\omega _Q\int_1^{\infty}{r^{-1-\varepsilon}}\int_{B\left( 0,r \right)}{\cdots \int_{B\left( 0,r \right)}{K\left( \left| y_1 \right|_h,...,\left| y_m \right|_h \right)}}\prod_{j=1}^m{| y_j |_{h}^{-\frac{\alpha _j}{q}-\frac{Q}{q_j}+\frac{\varepsilon}{q_j}}}dy_1\cdots dy_mdr
		\\
		&=-\frac{\omega _Q}{\varepsilon}\int_1^{\infty}{\left( r^{-\varepsilon} \right) ^{\prime}}\left( \omega _{Q}^{m}\int_0^r{\cdots}\int_0^r{K\left( r_1,...,r_m \right)}\prod_{j=1}^m{r_j^{-\frac{\alpha _j}{q}-\frac{Q}{q_j}+\frac{\varepsilon}{q_j}+Q-1}}dr_1\cdots dr_m \right) dr
		\\
		&=\frac{\omega _{Q}^{m+1}}{\varepsilon}\int_0^1{\cdots}\int_0^1{K\left( r_1,...,r_m \right)}\prod_{j=1}^m{r_j^{-\frac{\alpha _j}{q}-\frac{Q}{q_j}+\frac{\varepsilon}{q_j}+Q-1}}dr_1\cdots dr_m+\sum_{i=1}^m{L_i},
	\end{align*}
	where $L_i$ is defined as
	\begin{align*}
		L_i
		&=\frac{\omega _{Q}^{m+1}}{\varepsilon}\int_1^{\infty}{r^{-\varepsilon}}\int_0^r{\cdots}\int_0^r{K(r_1,...,\overset{\left( i \right)}{r},...,r_m)}r^{-\frac{\alpha _i}{q}-\frac{Q}{q_i}+\frac{\varepsilon}{q_i}+Q-1}
		\\
		&\quad \quad \quad \quad \quad \quad \quad \quad \quad \quad \times \prod_{j\ne i}^m{r_{j}^{-\frac{\alpha _j}{q}-\frac{Q}{q_j}+\frac{\varepsilon}{q_j}+Q-1}}dr_1...,\overset{\land}{dr_i}\cdots dr_mdr
		\\
		&=\frac{\omega _{Q}^{m+1}}{\varepsilon}\int_1^{\infty}{r_{i}^{-\varepsilon}}\int_0^{r_i}{\cdots}\int_0^{r_i}{K\left( r_1,...,r_m \right)}\prod_{j=1}^m{r_j^{-\frac{\alpha _j}{q}-\frac{Q}{q_j}+\frac{\varepsilon}{q_j}+Q-1}}dr_1\cdots \overset{\land}{dr_i}\cdots dr_mdr_i.
	\end{align*}
	Here $\overset{\land}{dr_i}$ means that we do not integrate with respect to the variable $r_i$. The last equality follows from integration by parts and the observation that, if we let
	$$F\left( x,...,x \right) =\int_0^x{\cdots}\int_0^x{K\left( r_1,...,r_m \right)}\prod_{j=1}^m{r_j^{-\frac{\alpha _j}{q}-\frac{Q}{q_j}+\frac{\varepsilon}{q_j}+Q-1}}dr_1\cdots dr_m,$$
	then 
	$$\frac{dF\left( x,...,x \right)}{dx}=\sum_{i=1}^m{\int_0^x{\cdots}\int_0^x{K( r_1,...,\overset{\left( i \right)}{x},... ,r_m)}}x^{-\frac{\alpha _i}{q}-\frac{Q}{q_i}+\frac{\varepsilon}{q_i}+Q-1}\times \prod_{j\ne i}^m{r_{j}^{-\frac{\alpha _j}{q}-\frac{Q}{q_j}+\frac{\varepsilon}{q_j}+Q-1}}dr_1\cdots \overset{\land}{dr_i}\cdots dr_m,$$
	where the upper index $(i)$ means that $x$ replaces the variable $r_i$ in the $i$-th position. By means of the previous step, we have
	\begin{equation}\label{3.3}
		\begin{aligned}
			&\frac{\left| \left< \mathcal{H}^h\left( f_{1,\varepsilon},...,f_{m,\varepsilon} \right) ,g_{\varepsilon} \right> \right|}{\left\| g_{\varepsilon} \right\| _{L^{q ^{\prime}}\left( \mathbb{H} ^n,\left| x \right|_{h}^{\alpha} \right)}\left\| f_{1,\varepsilon} \right\| _{L^{q_1}( \mathbb{H} ^n,\left| x \right|_{h}^{\frac{q_1\alpha _1}{q}} )}\cdots \left\| f_{m,\varepsilon} \right\| _{L^{q_m}( \mathbb{H} ^n,\left| x \right|_{h}^{\frac{q_m\alpha _m}{q}})}}
			\\
			&=\omega _{Q}^{m}\int_0^1{\cdots}\int_0^1{K\left( r_1,...,r_m \right)}\prod_{j=1}^m{r_{j}^{-\frac{\alpha _j}{q}-\frac{Q}{q_j}+\frac{\varepsilon}{q_j}+Q-1}}dr_1\cdots dr_m+\sum_{i=1}^m{\frac{\varepsilon L_i}{\omega _Q}}.
		\end{aligned}
	\end{equation}
	Let $E_i$ denote the domain of integral $L_i$ above (\ref{3.3}), that is,
	$$E_i=\left\{ \left( r_1,...,r_m \right) \in \left( 0,\infty \right) ^m:1\leqslant r_i<\infty ,0\leqslant r_j\leqslant r_i,j\ne i \right\}.$$
	Taking into account that $\frac{1}{q}=\frac{1}{q_1}+\cdots +\frac{1}{q_m}$, we can bound the integrand of $\frac{\varepsilon L_i}{\omega _Q}$ on $E_i$ as follows:
	\begin{align*}
		{r_i}^{-\varepsilon}K\left( r_1,...,r_m \right) \prod_{j=1}^m{r_{j}^{-\frac{\alpha _j}{q}-\frac{Q}{q_j}+\frac{\varepsilon}{q_j}+Q-1}}
		&\leqslant {r_i}^{-\varepsilon +\frac{\varepsilon}{q_1}+\cdots +\frac{\varepsilon}{q_m}}K\left( r_1,...,r_m \right) \prod_{j=1}^m{r_{j}^{-\frac{\alpha _j}{q}-\frac{Q}{q_j}+Q-1}}
		\\
		&={r_i}^{-\frac{\varepsilon}{q^{\prime}}}K\left( r_1,...,r_m \right) \prod_{j=1}^m{r_{j}^{-\frac{\alpha _j}{q}-\frac{Q}{q_j}+Q-1}}
		\\
		&\leqslant K\left( r_1,...,r_m \right) \prod_{j=1}^m{r_{j}^{-\frac{\alpha _j}{q}-\frac{Q}{q_j}+Q-1}}.
	\end{align*}
	For the integrand of the first term in (\ref{3.3}) on $\left[ 0,1 \right] ^m$, we also have that
	$$K\left( r_1,...,r_m \right) \prod_{j=1}^m{r_{j}^{-\frac{\alpha _j}{q}-\frac{Q}{q_j}+\frac{\varepsilon}{q_j}+Q-1}}\leqslant K\left( r_1,...,r_m \right) \prod_{j=1}^m{r_{j}^{-\frac{\alpha _j}{q}-\frac{Q}{q_j}+Q-1}}.$$
	Since the condition of kernel $K$ (\ref{3.1}) is equivence to
	\begin{equation}\label{3.4}
		\begin{aligned}
			C^h
			&=\int_{\mathbb{H} ^{mn}}{K\left( y_1,...,y_m \right)}\prod_{j=1}^m{| y_j |_{h}^{-\frac{\alpha _j}{q}-\frac{Q}{q_j}}}dy_1\cdots dy_m
			\\
			&=\omega _{Q}^{m}\int_0^{\infty}{\cdots}\int_0^{\infty}{K\left( r_1,...,r_m \right)}\prod_{j=1}^m{r_{j}^{-\frac{\alpha _j}{q}-\frac{Q}{q_j}+Q-1}}dr_1\cdots dr_m<\infty,
		\end{aligned}
	\end{equation}
	using assumption (\ref{3.4}), we can use the Lebesgue Dominated Convergence theorem, which implies that
	\begin{equation}\label{3.5}
		\begin{aligned}
			\underset{\varepsilon \rightarrow 0^+}{\lim}\frac{\varepsilon L_i}{\omega _Q}=\omega _{Q}^{m}\int_1^{\infty}{\int_0^{r_i}{\cdots}\int_0^{r_i}{K\left( r_1,...,r_m \right)}}\prod_{j=1}^m{r_{j}^{-\frac{\alpha _j}{q}-\frac{Q}{q_j}+Q-1}}dr_1\cdots dr_mdr_i
		\end{aligned}
	\end{equation}
	
	\noindent and
	\begin{equation}\label{3.6}
		\begin{aligned}
			&\underset{\varepsilon \rightarrow 0^+}{\lim}\omega _{Q}^{m}\int_0^1{\cdots}\int_0^1{K\left( r_1,...,r_m \right)}\prod_{j=1}^m{r_{j}^{-\frac{\alpha _j}{q}-\frac{Q}{q_j}+\frac{\varepsilon}{q_j}+Q-1}}dr_1\cdots dr_m
			\\
			&=\omega _{Q}^{m}\int_0^1{\cdots}\int_0^1{K\left( r_1,...,r_m \right)}\prod_{j=1}^m{r_{j}^{-\frac{\alpha _j}{q}-\frac{Q}{q_j}+Q-1}}dr_1\cdots dr_m.
		\end{aligned}
	\end{equation}
	Furthermore, we have 
	$$\left[ 0,1 \right] ^m\cup \left( \bigcup_{i=1}^m{E_i} \right) =\left( 0,\infty \right) ^m,$$
	and for $i,j=1,...,m$, any of the intersection  sets $\left[ 0,1 \right] ^m\cap E_i$, $E_i\cap E_j$, $i\ne j$, has Lebesgue measure zero in $\mathbb{H} ^n$. Consequently, (\ref{3.3}), (\ref{3.5}), and (\ref{3.6}) imply that
	\begin{align*}
		&\left\| \mathcal{H}^h \right\| _{L^{q_1}( \mathbb{H} ^n,\left| x \right|_{h}^{\frac{q_1\alpha _1}{q}}) \times \cdot \cdot \cdot \times L^{q_m}( \mathbb{H} ^n,\left| x \right|_{h}^{\frac{q_m\alpha _m}{q}}) \rightarrow L^q\left( \mathbb{H} ^n,\left| x \right|_{h}^{\alpha} \right)}
		\\
		&=\underset{\varepsilon \rightarrow 0^+}{\lim}\frac{\left| \left< \mathcal{H}^h\left( f_{1,\varepsilon},...,f_{m,\varepsilon} \right) ,g_{\varepsilon} \right> \right|}{\left\| g_{\varepsilon} \right\| _{L^{q ^{\prime}}( \mathbb{H} ^n,\left| x \right|_{h}^{\alpha})}\left\| f_{1,\varepsilon} \right\| _{L^{q_1}( \mathbb{H} ^n,\left| x \right|_{h}^{\frac{q_1\alpha _1}{q}} )}\cdots \left\| f_{m,\varepsilon} \right\| _{L^{q_m}( \mathbb{H} ^n,\left| x \right|_{h}^{\frac{q_m\alpha _m}{q}} )}}=C^h.
	\end{align*}
	This finishes the proof of Theorem \ref{thm3.1}.
\end{proof}\renewcommand{\qedsymbol}{}
\section{Application: Sharp weighted $\boldsymbol{L}^{\boldsymbol{p}}$ estimates for a class of integral operators }
By taking a particular kernel $K$ in operator $T$ defined by (\ref{3.1}), we can obtain sharp weighted $L^p$ estimates for Hardy, Hardy-Littlewood-P\'{o}lya, and Hilbert operators on the Heisenberg group. Our results in this section are as follows.
\begin{cor}\label{cor3.1}
	Assume that the real paramenters $q$, $q_j$, $\alpha$, and $\alpha_j$ with $j=1,2,...,m$ are the same as in Theorem \ref{thm3.1}, and $f_j$ is a radial function in $L^{q_{\boldsymbol{j}}}(\mathbb{H} ^n,\left| x \right|_{h}^{\frac{q_{\boldsymbol{j}}\alpha _{\boldsymbol{j}}}{q}})$. Assume also that $-\frac{\alpha _j}{q}-\frac{Q}{q_j}+Q>0$. Then
	\begin{equation}
		\begin{aligned}
			\left\| \mathcal{H}_1 ^h \right\| _{L^{q_1}(\mathbb{H} ^n,\left| x \right|_{h}^{\frac{q_1\alpha _1}{q}})\times \cdot \cdot \cdot \times L^{q_m}(\mathbb{H} ^n,\left| x \right|_{h}^{\frac{q_m\alpha _m}{q}})\rightarrow L^q(\mathbb{H} ^n,\left| x \right|_{h}^{\alpha})}=\frac{q\omega _Q2^{1-m}}{mqQ-\alpha -Q}\frac{\prod\nolimits_{j=1}^m{\varGamma ( ( Q-\frac{\alpha _j}{q}-\frac{Q}{q_j} ) /2 )}}{\varGamma ( ( mQ-\frac{\alpha}{q}-\frac{Q}{q} ) /2 )}.
		\end{aligned}
	\end{equation}
\end{cor}
\begin{cor}\label{cor3.2}
	Assume that the real paramenters $q$, $q_j$, $\alpha$, and $\alpha_j$ with $j=1,2,...,m$ are the same as in Theorem \ref{thm3.1}, $f_j$ is a radial function in $L^{q_{\boldsymbol{j}}}(\mathbb{H} ^n,\left| x \right|_{h}^{\frac{q_{\boldsymbol{j}}\alpha _{\boldsymbol{j}}}{q}})$. Assume also that $-\frac{\alpha _j}{q}-\frac{Q}{q_j}+Q>0$ and $-\frac{\alpha}{q}-\frac{Q}{q}<0$. Then
	\begin{equation}
		\begin{aligned}
			\left\| \mathcal{H} _{2}^{h} \right\| _{L^{q_1}(\mathbb{H} ^n,\left| x \right|_{h}^{\frac{q_1\alpha _1}{q}})\times \cdot \cdot \cdot \times L^{q_m}(\mathbb{H} ^n,\left| x \right|_{h}^{\frac{q_m\alpha _m}{q}})\rightarrow L^q(\mathbb{H} ^n,\left| x \right|_{h}^{\alpha})}=\frac{q\omega _{Q}^{m}( mQ-\frac{\alpha}{q}-\frac{Q}{q})}{\left( \alpha +Q \right) \prod_{j=1}^m{( Q-\frac{\alpha _j}{q}-\frac{Q}{q_j} )}}.
		\end{aligned}
	\end{equation}
\end{cor}
\begin{cor}\label{cor3.3}
	Assume that the real paramenters $q$, $q_j$, $\alpha$, and $\alpha_j$ with $j=1,2,...,m$ are the same as in Theorem \ref{thm3.1}, and $f_j$ is a radial function in $L^{q_{\boldsymbol{j}}}(\mathbb{H} ^n,\left| x \right|_{h}^{\frac{q_{\boldsymbol{j}}\alpha _{\boldsymbol{j}}}{q}})$. Assume also that $-\frac{\alpha _j}{q}-\frac{Q}{q_j}+Q>0$ and $-\frac{\alpha}{q}-\frac{Q}{q}<0$. Then
	\begin{equation}
		\begin{aligned}
			\left\| \mathcal{H} _{3}^{h} \right\| _{L^{q_1}(\mathbb{H} ^n,\left| x \right|_{h}^{\frac{q_1\alpha _1}{q}})\times \cdot \cdot \cdot \times L^{q_m}(\mathbb{H} ^n,\left| x \right|_{h}^{\frac{q_m\alpha _m}{q}})\rightarrow L^q(\mathbb{H} ^n,\left| x \right|_{h}^{\alpha})}=\frac{\omega _{Q}^{m}\varGamma ( \frac{\alpha +Q}{qQ} ) \prod_{j=1}^m{\varGamma ( 1-\frac{\alpha _j}{qQ}-\frac{1}{q_j})}}{Q^m\varGamma \left( m \right)}\,\,.
		\end{aligned}
	\end{equation}
\end{cor}
\begin{proof}[Proof of Corollary \ref{cor3.1}.]\renewcommand{\qedsymbol}{}
	Next, we will use the methods in \cite{Hardy1, Hardy2}. If we take the kernel
	\begin{align}
		K\left( y_1,...,y_m \right) =\chi _{\{ \left| \left( y_1,...,y_m \right) \right|_h\leqslant 1 \}}\left( y_1,...,y_m \right)\end{align}
	in Theorems \ref{thm3.1}, by a change of variables, it is easy to verify that $\mathcal{H}^h=\mathcal{H}_1^h$, and then $\mathcal{H}_1^h$ can be denoted by 
	$$\mathcal{H}_1^h=\int_{\left| \left( y_1,...,y_m \right) \right|_h\leqslant 1}{f_1( \left| x \right|_{h}y_1 ) \cdots f_m( \left| x \right|_{h}y_m ) dy_1\cdots dy_m}.$$
	Then all things reduce to calculating
	$$C_{1}^{h}=\int_{\left| \left( y_1,...,y_m \right) \right|_h\leqslant 1}{\prod_{i=1}^m{\left| y_i \right|_{h}^{-\frac{\alpha _j}{q}-\frac{Q}{q_j}}}dy_1\cdots dy_m}.$$
	To calculate this integral, employing the polar coordinates $y_j=\rho _j\xi _j$, $j=1,2,...,m$, and Fubini's theorem, we obtain
	\begin{equation}\label{3.10}
		\begin{aligned}
			C_{1}^{h}
			&=\int_{\mathbb{S}}{\cdot \cdot \cdot \int_{\mathbb{S}}{\int_{\sum_{j=1}^m{\rho _{j}^{2}<1,\rho _j>0,j=1,...,m}}{\prod_{j=1}^m{\rho_{j}^{-\frac{\alpha _j}{q}-\frac{Q}{q_j}+Q-1}}}}}d\rho _1\cdots d\rho _md\sigma \left( \xi _1 \right) \cdot \cdot \cdot d\sigma \left( \xi _m \right) 
			\\
			&=\omega _{Q}^{m}\int_{\sum_{j=1}^m{\rho _{j}^{2}<1,\rho _j>0,j=1,...,m}}{\prod_{j=1}^m{\rho_{j}^{-\frac{\alpha _j}{q}-\frac{Q}{q_j}+Q-1}}}d\rho _1\cdots d\rho _m.
		\end{aligned}
	\end{equation}
	We use the $m$-dimensional spherical coordinates
	\begin{align*}
		&\rho _1=r\cos \varphi _1,
		\\
		&\rho _2=r\sin \varphi _1\cos \varphi _2,
		\\
		&\rho _3=r\sin \varphi _1\sin \varphi _2\cos \varphi _3,
		\\
		&\vdots 
		\\
		&\rho _{m-1}=r\sin \varphi _1\sin \varphi _2\cdots \sin \varphi _{m-2}\cos \varphi _{m-1},
		\\
		&\rho _m=r\sin \varphi _1\sin \varphi _2\cdots \sin \varphi _{m-2}\sin \varphi _{m-1},
	\end{align*}
	where $r\geqslant 0$ is the radial coordinate and $\varphi _j$, $j=1,2,...,m-1$, are angular coordinates, $\varphi _j\in \left[ 0,\pi \right]$, $j=1,2,...,m-2$, $\varphi _{m-1}\in \left[0,2\pi \right)$, and the known fact that the associated Jacobian is
	$$\left| J_m \right|=r^{m-1}\sin ^{m-2}\varphi _1\sin ^{m-3}\varphi _2\cdots \sin \varphi _{m-2}$$
	in (\ref{3.10}). Since $-\frac{\alpha _j}{q}-\frac{Q}{q_j}+Q>0$, we have
	\begin{align*}
		&C_{1}^{h}=\omega _{Q}^{m}\int_0^1{r^{\sum_{j=1}^m{( -\frac{\alpha _j}{q}-\frac{Q}{q_j}+Q-1 )}}}r^{m-1}\int_0^{\frac{\pi}{2}}{\int_0^{\frac{\pi}{2}}{\cdots \int_0^{\frac{\pi}{2}}{( \sin \varphi _1 ) ^{m-2}( \sin \varphi _2 ) ^{m-3}\cdots ( \sin \varphi _{m-2} ) ^1}}}
		\\
		&( \sin \varphi _1) ^{\sum_{j=2}^m{( -\frac{\alpha _j}{q}-\frac{Q}{q_j}+Q-1)}}( \sin \varphi _2) ^{\sum_{j=3}^m{( -\frac{\alpha _j}{q}-\frac{Q}{q_j}+Q-1 )}}\cdots ( \sin \varphi _{m-2}) ^{\sum_{j=m-1}^m{( -\frac{\alpha _j}{q}-\frac{Q}{q_j}+Q-1)}}( \sin \varphi _{m-1} ) ^{-\frac{\alpha _m}{q}-\frac{Q}{q_m}+Q-1}
		\\
		&( \cos \varphi _1) ^{-\frac{\alpha _1}{q}-\frac{Q}{q_1}+Q-1}\cdots ( \cos \varphi _{m-2} ) ^{-\frac{\alpha _{m-2}}{q}-\frac{Q}{q_{m-2}}+Q-1}( \cos \varphi _{m-1}) ^{-\frac{\alpha _{m-1}}{q}-\frac{Q}{q_{m-1}}+Q-1}d\varphi _1\cdots d\varphi _{m-1}dr
		\\
		&=\omega _{Q}^{m}\int_0^1{r^{\sum_{j=1}^m{( -\frac{\alpha _j}{q}-\frac{Q}{q_j}+Q) -1}}}dr
		\\
		&\quad\quad\quad\quad\quad\quad\quad\quad\times\int_0^{\frac{\pi}{2}}{\cdots \int_0^{\frac{\pi}{2}}{\prod_{j=1}^{m-1}{( \sin \varphi _j ) ^{m-( j+1) +\sum_{i=j+1}^m{( -\frac{\alpha _i}{q}-\frac{Q}{q_i}+Q-1)}}}}}( \cos \varphi _j) ^{-\frac{\alpha _j}{q}-\frac{Q}{q_j}+Q-1}d\varphi _1\cdots d\varphi _{m-1}
		\\
		&=\frac{\omega _{Q}^{m}}{\sum_{j=1}^m{( -\frac{\alpha _j}{q}-\frac{Q}{q_j}+Q )}}\int_0^{\frac{\pi}{2}}{\cdots \int_0^{\frac{\pi}{2}}{\prod_{j=1}^{m-1}{( \sin \varphi _j ) ^{Q( m-j ) -1-\sum_{i=j+1}^m{( \frac{\alpha _i}{q}+\frac{Q}{q_i} )}}}}}( \cos \varphi _j ) ^{-\frac{\alpha _j}{q}-\frac{Q}{q_j}+Q-1}d\varphi _1\cdots d\varphi _m
		\\
		&=\frac{\omega _{Q}^{m}}{mQ-\frac{\alpha +Q}{q}}\prod_{j=1}^{m-1}{\int_0^1{t_{j}^{Q( m-j) -1-\sum_{i=j+1}^m{( \frac{\alpha _i}{q}+\frac{Q}{q_i} )}}}}( 1-t_{j}^{2} ) ^{\frac{Q-2-\frac{\alpha _j}{q}-\frac{Q}{q_j}}{2}}dt_j
		\\
		&=\frac{q\omega _{Q}^{m}2^{1-m}}{mqQ-\alpha -Q}\prod_{j=1}^{m-1}{\int_0^1{s_{j}^{\frac{Q(m-j)-\sum_{i=j+1}^m{(\frac{\alpha _i}{q}+\frac{Q}{q_i})}}{2}-1}}}\left( 1-s_{j\,\,} \right) ^{\frac{Q-\frac{\alpha _j}{q}-\frac{Q}{q_j}}{2}-1}ds_j
		\\
		&=\frac{q\omega _{Q}^{m}2^{1-m}}{mqQ-\alpha -Q}\prod_{j=1}^{m-1}{B\left( \frac{Q( m-j ) -\sum_{i=j+1}^m{( \frac{\alpha _i}{q}+\frac{Q}{q_i} )}}{2},\frac{Q-\frac{\alpha _j}{q}-\frac{Q}{q_j}}{2} \right)}.
	\end{align*}
	We also use the fact $\varphi _j\in \left( 0,\pi /2 \right)$, $j=1,2,...,m-1$. From the seventh to eighth lines, we let $t_j=\sin \varphi _j$ for $j=1,2,...,m-1$. From the eighth to ninth lines, we set $s_j=t_j^2$ for $j=1,2,...,m-1$, as well as the definition of the beta function (see, e.g., \cite{Hardylunwenji}).
	
	By using the following well-known relation between Euler's beta and gamma functions:
	$$B\left( a,b \right) =\frac{\varGamma \left( a \right) \varGamma \left( b \right)}{\varGamma \left( a+b \right)},$$
	(see, for example, \cite{Hardylunwenji}), after some simple calculations, we see that the following relations hold:
	\begin{align*}
		\prod_{j=1}^{m-1}{B\left( \frac{Q\left( m-j \right) -\sum_{i=j+1}^m{( \frac{\alpha _i}{q}+\frac{Q}{q_i} )}}{2},\frac{Q-\frac{\alpha _j}{q}-\frac{Q}{q_j}}{2} \right)}
		&=\prod_{j=1}^{m-1}{\left( \frac{\varGamma \left( \frac{Q\left( m-j \right) -\sum_{i=j+1}^m{( \frac{\alpha _i}{q}+\frac{Q}{q_i})}}{2} \right) \varGamma \left( \frac{Q-\frac{\alpha _j}{q}-\frac{Q}{q_j}}{2} \right)}{\varGamma \left( \frac{Q\left( m-\left( j-1 \right) \right) -\sum_{i=j}^m{( \frac{\alpha _i}{q}+\frac{Q}{q_i})}}{2} \right)} \right)}
		\\
		&=\frac{\prod_{j=1}^m{\varGamma \left( \frac{Q-\frac{\alpha _j}{q}-\frac{Q}{q_j}}{2} \right)}}{\varGamma \left( \frac{mQ-\frac{\alpha}{q}-\frac{Q}{q}}{2} \right)}.
	\end{align*}
	Then we have
	$$C_{1}^{h}=\frac{q\omega _{Q}^{m}2^{1-m}}{mqQ-\alpha -Q}\frac{\prod_{j=1}^m{\varGamma ( ( Q-\frac{\alpha _j}{q}-\frac{Q}{q_j}) /2 )}}{\varGamma ( ( mQ-\frac{\alpha}{q}-\frac{Q}{q} ) /2)}.$$
	This finishes the proof of Corollary \ref{cor3.1}.
\end{proof}\renewcommand{\qedsymbol}{}
\begin{proof}[Proof of Corollary \ref{cor3.2}.]\renewcommand{\qedsymbol}{}
	Next, we refer to the methods in \cite{HLP}. If we take the kernel 
	\begin{align}
		K\left( y_1,...,y_m \right) =\frac{1}{[\max\mathrm{(}1,\left| y_1 \right|_{h}^{Q},...,\left| y_m \right|_{h}^{Q})]^m}
	\end{align}
	in Theorem \ref{thm2.1}, by a change of variables, we have $\mathcal{H}^p=\mathcal{H}_{2}^{h}$, and then $\mathcal{H}_{2}^{h}$ can be denoted by 
	$$\mathcal{H} _{2}^{h}=\int_{\mathbb{H} ^{mn}}{\frac{1}{[\max\mathrm{(}1,\left| y_1 \right|_{h}^{Q},...,\left| y_m \right|_{h}^{Q})]^m}f_1(\left| x \right|_hy_1)\cdots f_m(\left| x \right|_hy_m)dy_1\cdots dy_m}.$$
	Then, we reduce to calculating
	$$C_{2}^{h}=\int_{\mathbb{H} ^{mn}}{\frac{1}{[\max\mathrm{(}1,\left| y_1 \right|_{h}^{Q},...,\left| y_m \right|_{h}^{Q})]^m}\prod_{j=1}^m{| y_j |_{h}^{-\frac{\alpha _j}{q}-\frac{Q}{q_j}}}dy_1\cdots dy_m}.$$
	To calculate this integral, we divide the integral into $m$ parts. Let
	\begin{align*}
		&E_0=\left\{ \left( y_1,...,y_m \right) \in \mathbb{H} ^n\times \cdots \times \mathbb{H} ^n:\left| y_k \right|_h\leqslant 1,1\leqslant k\leqslant m \right\} ;
		\\
		&E_1=\left\{ \left( y_1,...,y_m \right) \in \mathbb{H} ^n\times \cdots \times \mathbb{H} ^n:\left| y_1 \right|_h>1,\left| y_k \right|_h\leqslant \left| y_1 \right|_h,2\leqslant k\leqslant m \right\} ;
		\\
		&E_i=\{ \left( y_1,...,y_m \right) \in \mathbb{H} ^n\times \cdots \times \mathbb{H} ^n:| y_i |_h>1,| y_j|_h<\left| y_i \right|_h,\left| y_k \right|_h\leqslant \left| y_i \right|_h,1\leqslant j<i<k\leqslant m \} ;
		\\
		&E_m=\{ \left( y_1,...,y_m \right) \in \mathbb{H} ^n\times \cdots \times \mathbb{H} ^n:\left| y_m \right|_h>1,| y_j |_h<\left| y_m \right|_h,1\leqslant j<m \}.
	\end{align*}
	It is clear that 
	$$\bigcup_{j=0}^m{E_j=\mathbb{H} ^n\times \cdots \times \mathbb{H} ^n},$$
	and $E_i\cap E_j=\varnothing \left( i\ne j \right) $.
	Let
	$$K_j:=\int_{E_j}{\frac{1}{[\max\mathrm{(}1,\left| y_1 \right|_{h}^{Q},...,\left| y_m \right|_{h}^{Q})]^m}\prod_{k=1}^m{\left| y_k \right|_{h}^{-\frac{\alpha _k}{q}-\frac{Q}{q_k}}}}dy_1\cdots dy_m,$$
	and then we have
	$$C_{2}^{h}=\sum_{j=1}^m{K_j:}=\sum_{j=1}^m{\int_{E_j}{\frac{1}{[\max\mathrm{(}1,\left| y_1 \right|_{h}^{Q},...,\left| y_m \right|_{h}^{Q})]^m}\prod_{k=1}^m{\left| y_k \right|_{h}^{-\frac{\alpha _j}{q}-\frac{Q}{q_j}}}dy_1\cdots dy_m}}.$$
	Now let us calculate $J_j$, $j=1,2,...,m$. Since $-\frac{\alpha _j}{q}-\frac{Q}{q_j}+Q>0$, using the polar coordinate transformation, we have 
	\begin{align*}
		K_0&=\int_{E_0}{\prod_{j=1}^m{\left| y_i \right|_{h}^{-\frac{\alpha _j}{q}-\frac{Q}{q_j}}}}dy_1\cdots dy_m=\prod_{j=1}^m{\int_{\left| y_i \right|_h\leqslant 1}{\left| y_i \right|_{h}^{-\frac{\alpha _j}{q}-\frac{Q}{q_j}}dy_i}}
		\\
		&=\prod_{j=1}^m{\omega _Q\int_0^1{r_{j}^{-\frac{\alpha _j}{q}-\frac{Q}{q_j}+Q-1}dr_j}=\frac{\omega _{Q}^{m}}{\prod_{j=1}^m{( Q-\frac{\alpha _j}{q}-\frac{Q}{q_j})}}}.
	\end{align*}
	For $j=1$, since $-\frac{\alpha _j}{q}-\frac{Q}{q_j}+Q>0$ and $-\frac{\alpha}{q}-\frac{Q}{q}<0$, we have
	\begin{align*}
		K_1&=\int_{E_1}{\frac{\prod_{j=1}^m{\left| y_i \right|_{h}^{-\frac{\alpha _j}{q}-\frac{Q}{q_j}}}}{[ \max( 1,\left| y_1 \right|_{h}^{Q},...\left| y_m \right|_{h}^{Q})] ^m}}dy_1\cdots dy_m=\int_{E_1}{\left| y_1 \right|_{h}^{-\frac{\alpha _1}{q}-\frac{Q}{q_1}-mQ}}\prod_{j=2}^m{| y_j |_{h}^{-\frac{\alpha _j}{q}-\frac{Q}{q_j}}}dy_2\cdots dy_mdy_1
		\\
		&=\int_{\left| y_1 \right|_h>1}{\left| y_1 \right|_{h}^{-\frac{\alpha _1}{q}-\frac{Q}{q_1}-mQ}}\prod_{j=2}^m{\int_{\left| y_j \right|_h\leqslant \left| y_1 \right|_h}{| y_j|_{h}^{-\frac{\alpha _j}{q}-\frac{Q}{q_j}}dy_j}}dy_1
		\\
		&=\int_{\left| y_1 \right|_h>1}{\left| y_1 \right|_{h}^{-\frac{\alpha _1}{q}-\frac{Q}{q_1}-mQ}}\prod_{j=2}^m{\omega _Q}\frac{\left| y_1 \right|_{h}^{Q-\frac{\alpha _j}{q}-\frac{Q}{q_j}}}{Q-\frac{\alpha _j}{q}-\frac{Q}{q_j}}dy_1
		\\
		&=\frac{\omega _{Q}^{m-1}}{\prod_{j=2}^m{( Q-\frac{\alpha _j}{q}-\frac{Q}{q_j})}}\int_{\left| y_1 \right|_h>1}{\left| y_1 \right|_{h}^{-\frac{\alpha}{q}-\frac{Q}{q}-Q}}dy_1=\frac{q\omega _{Q}^{m}}{\left( \alpha +Q \right) \prod_{j=2}^m{( Q-\frac{\alpha _j}{q}-\frac{Q}{q_j})}}.
	\end{align*}
	Similar for $i=2,...,m-1$, we have
	\begin{align*}
		K_i&=\int_{E_i}{\left| y_i \right|_{h}^{-\frac{\alpha _i}{q}-\frac{Q}{q_i}-mQ}}\prod_{j\ne i}^m{| y_j |_{h}^{-\frac{\alpha _j}{q}-\frac{Q}{q_j}}}dy_1\cdots dy_{i-1}dy_{i+1}dy_mdy_i
		\\
		&=\int_{\left| y_i \right|_h>1}{\left| y_i \right|_{h}^{-\frac{\alpha _i}{q}-\frac{Q}{q_i}-mQ}}\prod_{j=2}^{i-1}{\int_{| y_j|_h\leqslant \left| y_i \right|_h}{| y_j |_{h}^{-\frac{\alpha _j}{q}-\frac{Q}{q_j}}dy_j}}\prod_{k=i+1}^m{\int_{\left| y_k \right|_h\leqslant \left| y_i \right|_h}{\left| y_k \right|_{h}^{-\frac{\alpha _k}{q}-\frac{Q}{q_k}}dy_kdy_i}}
		\\
		&=\frac{q\omega _{Q}^{m}}{\left( \alpha +Q \right) \prod_{1\leqslant j\leqslant m,j\ne i}{( Q-\frac{\alpha _j}{q}-\frac{Q}{q_j})}}.
	\end{align*}
	When $i=m$, similar to the previous step, we show that
	$$K_m=\int_{\left| y_i \right|_h>1}{\left| y_m \right|_{h}^{-\frac{\alpha _m}{q}-\frac{Q}{q_m}-mQ}}\prod_{j=1}^{m-1}{| y_j |_{h}^{-\frac{\alpha _j}{q}-\frac{Q}{q_j}}}dy_1\cdots dy_m=\frac{q\omega _{Q}^{m}}{\left( \alpha +Q \right) \prod_{j=1}^{m-1}{( Q-\frac{\alpha _j}{q}-\frac{Q}{q_j})}}.$$
	Then, it yields that
	$$C_{h}^{2}=K_0+K_1+\sum_{i=2}^{m-1}{K_i}+K_m=\frac{q\omega _{Q}^{m}(mQ-\frac{\alpha}{q}-\frac{Q}{q})}{\left( \alpha +Q \right) \prod_{j=1}^m{(Q-\frac{\alpha _j}{q}-\frac{Q}{q_j})}}.$$
	This finishes the proof of Corollary \ref{cor3.2}.
\end{proof}\renewcommand{\qedsymbol}{}
\begin{proof}[Proof of Corollary \ref{cor3.3}.]\renewcommand{\qedsymbol}{}
	If we take the kernel 
	\begin{align}
		K\left( y_1,...,y_m \right) =\frac{1}{(1+\left| y_1 \right|_{h}^{Q}+\cdots +\left| y_m \right|_{h}^{Q})^m}
	\end{align}
	in Theorem \ref{thm2.1}, by a change of variables, we have $\mathcal{H}^p=\mathcal{H}_{3}^{h}$, and then $\mathcal{H}_{3}^{h}$ can be denoted by 
	$$\mathcal{H} _{3}^{h}=\int_{\mathbb{H} ^{mn}}{\frac{1}{(1+\left| y_1 \right|_{h}^{Q}+\cdots +\left| y_m \right|_{h}^{Q})^m}f_1(\left| x \right|_hy_1)\cdots f_m(\left| x \right|_hy_m)dy_1\cdots dy_m}.$$
	Then, we reduce to calculating
	$$C_{3}^{h}=\int_{\mathbb{H} ^{mn}}{\frac{1}{(1+\left| y_1 \right|_{h}^{Q}+\cdots +\left| y_m \right|_{h}^{Q})^m}\prod_{j=1}^m{| y_j |_{h}^{-\frac{\alpha _j}{q}-\frac{Q}{q_j}}}dy_1\cdots dy_m}.$$
	Actually, this method stems from Benyi and Oh \cite{Hilbert}, who investigated the one-dimensional case. Following their method, it is easy to find the higher-dimensional case, as well. For completeness, we give the details. Employing the polar coordinates and making a change of variables, we have
	$$C_{3}^{h}=\omega _{Q}^{m}\int_0^{\infty}{\cdots \int_0^{\infty}{\frac{1}{(1+\rho _{1}^{Q}+\cdots +\rho _{m}^{Q})^m}\prod_{j=1}^m{\rho_{j}^{-\frac{\alpha _j}{q}-\frac{Q}{q_j}+Q-1}}}}d\rho _1\cdots d\rho _m.$$
	Let $\rho _{j}^{Q}=t_j$, and we have
	$$C_{3}^{h}=\frac{\omega _{Q}^{m}}{Q^m}\int_0^{\infty}{\cdots \int_0^{\infty}{\frac{1}{\left( 1+t_1+\cdots +t_m \right) ^m}\prod_{j=1}^m{t_{j}^{\frac{-\frac{\alpha _j}{q}-\frac{Q}{q_j}+Q}{Q}-1}}}}dt_1\cdots dt_m.$$
	Let us denote the integral on the right by 
	$$\frac{Q^m}{\omega _{Q}^{m}}C_{3}^{h}=I_m\left( m,\frac{Q-\frac{\alpha _1}{q}-\frac{Q}{q_1}}{Q},...,\frac{Q-\frac{\alpha _m}{q}-\frac{Q}{q_m}}{Q} \right).$$
	By making the change of variables $t_m=\left( 1+t_1+\cdots +t_{m-1} \right) t$ and performing integration with respect to $dt$, we have the following identity:
	\begin{align*}
		&I_m\left( m,\frac{Q-\frac{\alpha _1}{q}-\frac{Q}{q_1}}{Q},...,\frac{Q-\frac{\alpha _m}{q}-\frac{Q}{q_m}}{Q} \right) 
		\\
		&=\int_0^{\infty}{\cdots \int_0^{\infty}{\frac{\left[ \left( 1+t_1+\cdots +t_{m-1} \right) t \right] ^{\frac{-\frac{\alpha _m}{q}-\frac{Q}{q_m}+Q}{Q}-1}\left( 1+t_1+\cdots +t_{m-1} \right)}{\left[ \left( 1+t_1+\cdots +t_{m-1} \right) +\left( 1+t_1+\cdots +t_{m-1} \right) t \right] ^m}}\prod_{j=1}^{m-1}{t_{j}^{\frac{-\frac{\alpha _j}{q}-\frac{Q}{q_j}+Q}{Q}-1}dt_1}}\cdots dt_{m-1}dt
		\\
		&=\int_0^{\infty}{\left( 1+t \right) ^{-m}}t^{\frac{-\frac{\alpha _m}{q}-\frac{Q}{q_m}+Q}{Q}-1}dtI_{m-1}\left( m-\frac{Q-\frac{\alpha _m}{q}-\frac{Q}{q_m}}{Q},\frac{Q-\frac{\alpha _1}{q}-\frac{Q}{q_1}}{Q},...,\frac{Q-\frac{\alpha _m}{q}-\frac{Q}{q_m}}{Q} \right).
	\end{align*}
	Observe that, if we make the change of variables $t+1=1/s$, we obtain
	$$\int_0^{\infty}{\left( 1+t \right) ^{-a-b}}t^{a-1}dt=\int_0^1{s^{a-1}\left( 1-s \right) ^{b-1}}ds=B\left( a,b \right) =\frac{\varGamma \left( a \right) \varGamma \left( b \right)}{\varGamma \left( a+b \right)}$$
	for $a,b>0$.
	According to known conditions, obviously, we have $( Q-\frac{\alpha _j}{q}-\frac{Q}{q_j}) /Q>0$. Now, we will show that $m-( Q-\frac{\alpha _m}{q}-\frac{Q}{q_m}) /Q-\cdots -( Q-\frac{\alpha _k}{q}-\frac{Q}{q_k}) /Q>0$ with $k=m,m-1,...,1$. Since $-\frac{\alpha _j}{q}-\frac{Q}{q_j}+Q>0$ and $-\frac{\alpha}{q}-\frac{Q}{q}<0$, we have
	\begin{align*}
		m-\frac{Q-\frac{\alpha _m}{q}-\frac{Q}{q_m}}{Q}-\cdots -\frac{Q-\frac{\alpha _k}{q}-\frac{Q}{q_k}}{Q}
		&=\frac{\left( k-1 \right) Q+( \frac{\alpha _m}{q}+\frac{Q}{q_m} ) +\cdots +( \frac{\alpha _k}{q}+\frac{Q}{q_k} )}{Q}
		\\
		&=\frac{\left( k-1 \right) Q+( \frac{\alpha}{q}+\frac{Q}{q}) -( \frac{\alpha _{k-1}}{q}+\frac{Q}{q_{k-1}}) -\cdots- ( \frac{\alpha _1}{q}+\frac{Q}{q_1})}{Q}
		\\
		&=\frac{( Q-\frac{\alpha _{k-1}}{q}-\frac{Q}{q_{k-1}}) +\cdots +( Q-\frac{\alpha _1}{q}-\frac{Q}{q_1}) +( \frac{\alpha}{q}+\frac{Q}{q})}{Q}>0.
	\end{align*}
	Therefore, if we recall the relationship between the beta and gamma functions, we obtain
	$$\int_0^{\infty}{\left( 1+t \right) ^{-m}}t^{( \frac{\alpha _m}{q}-\frac{Q}{q_m}+Q) /Q-1}dt=\frac{\varGamma( ( \frac{\alpha _m}{q}-\frac{Q}{q_m}+Q) /Q ) \varGamma ( m-( \frac{\alpha _m}{q}-\frac{Q}{q_m}+Q ) /Q)}{\varGamma \left( m \right)},$$
	and
	\begin{align*}
		&I_m\left( m,\frac{Q-\frac{\alpha _1}{q}-\frac{Q}{q_1}}{Q},...,\frac{Q-\frac{\alpha _m}{q}-\frac{Q}{q_m}}{Q} \right) 
		\\
		&=\frac{\varGamma \left( \frac{\frac{\alpha _m}{q}-\frac{Q}{q_m}+Q}{Q} \right) \varGamma \left( m-\frac{\frac{\alpha _m}{q}-\frac{Q}{q_m}+Q}{Q} \right)}{\varGamma \left( m \right)}I_{m-1}\left( m-\frac{Q-\frac{\alpha _m}{q}-\frac{Q}{q_m}}{Q},\frac{Q-\frac{\alpha _1}{q}-\frac{Q}{q_1}}{Q},...,\frac{Q-\frac{\alpha _{m-1}}{q}-\frac{Q}{q_{m-1}}}{Q} \right).
	\end{align*}
	By a simple induction argument, we obtain from this recurrence that
	\begin{align*}
		I\left( m,\frac{Q-\frac{\alpha _1}{q}-\frac{Q}{q_1}}{Q},...,\frac{Q-\frac{\alpha _m}{q}-\frac{Q}{q_m}}{Q} \right) 
		&=\frac{\varGamma \left( m-\frac{Q-\frac{\alpha _m}{q}-\frac{Q}{q_m}}{Q}-\cdots -\frac{Q-\frac{\alpha _1}{q}-\frac{Q}{q_1}}{Q} \right) \varGamma \left( \frac{Q-\frac{\alpha _m}{q}-\frac{Q}{q_m}}{Q} \right) \cdots \varGamma \left( \frac{Q-\frac{\alpha _1}{q}-\frac{Q}{q_1}}{Q} \right)}{\varGamma \left( m \right)}
		\\
		&=\frac{\varGamma \left( \frac{\alpha +Q}{qQ} \right) \prod_{j=1}^m{\varGamma \left( 1-\frac{\alpha _j}{qQ}-\frac{1}{q_j} \right)}}{\varGamma \left( m \right)}.
	\end{align*}
	Then
	$$C_{3}^{h}=\frac{\omega _{Q}^{m}}{Q^m}I_m\left( m,\frac{Q-\frac{\alpha _1}{q}-\frac{Q}{q_1}}{Q},...,\frac{Q-\frac{\alpha _m}{q}-\frac{Q}{q_m}}{Q} \right) =\frac{\omega _{Q}^{m}\varGamma \left( \frac{\alpha +Q}{qQ} \right) \prod_{j=1}^m{\varGamma \left( 1-\frac{\alpha _j}{qQ}-\frac{1}{q_j} \right)}}{Q^m\varGamma \left( m \right)}\,\,.$$
	This finishes the proof of Corollary \ref{cor3.3}.
\end{proof}\renewcommand{\qedsymbol}{}
\section{Further results: Sharp weighted $\boldsymbol{L}^{\boldsymbol{p}}$ estimate for the Hausdorff operator}
In this section, we will use the previous results to give the weighted $L^p$ estimates for the $m$-linear $n$-dimensional Hausdorff operator on the Heisenberg group.
\begin{cor}\label{cor5.1}
	Assume that the real paramenters $q$, $q_j$, $\alpha$, and $\alpha_j$ with $j=1,2,...,m$ are the same as in Theorem \ref{thm3.1}. A nonnegative function $\Phi$ on $\mathbb{H} ^n$ satisfies
	
	\begin{align}
		C_{\Phi}^{h}=\int_0^{\infty}{\cdots \int_0^{\infty}{\int_{\mathbb{S} ^{Q-1}}{\cdots \int_{\mathbb{S} ^{Q-1}}{\frac{\Phi \left( \delta _{r_1}y_{1}^{\prime},...,\delta _{r_m}y_{m}^{\prime} \right)}{\left| r_1 \right|_{h}^{Q}\cdots \left| r_m \right|_{h}^{Q}}\prod_{j=1}^m{r_{j}^{\frac{\alpha _j}{q}+\frac{Q}{q_j}-\frac{\varepsilon}{q_j}-Q+1}}}}}}dy_{1}^{\prime}\cdots dy_{m}^{\prime}dr_1\cdots dr_m<\infty.
	\end{align}
	Then 
	\begin{align}
		\left\| \mathcal{H} _{\Phi}^{h} \right\| _{L^{q_1}(\mathbb{H} ^n,\left| x \right|_{h}^{\frac{q_1\alpha _1}{q}})\times \cdot \cdot \cdot \times L^{q_m}(\mathbb{H} ^n,\left| x \right|_{h}^{\frac{q_m\alpha _m}{q}})\rightarrow L^q(\mathbb{H} ^n,\left| x \right|_{h}^{\alpha})}=C_{\Phi}^{h}.
	\end{align}
\end{cor}
\begin{proof}[Proof.]\renewcommand{\qedsymbol}{}
	By a change of variables, the $m$-linear $n$-dimensional Hausdorff operator become
	$$\mathcal{H} _{\Phi}^{h}=\int_{\mathbb{H} ^n}{\cdots}\int_{\mathbb{H} ^n}{\frac{\Phi \left( y_1,...,y_m \right)}{\left| y_1 \right|_{h}^{n}\cdots \left| y_m \right|_{h}^{n}}}f_1(\delta _{\left| y_1 \right|_{h}^{-1}}x)\cdots f_m(\delta _{\left| y_m \right|_{h}^{-1}}x)dy_1\cdots dy_m.$$
	We can obtain
	\begin{align*}
		&\left\| \mathcal{H} _{\Phi}^{h} \right\| _{L^{q_1}(\mathbb{H} ^n,\left| x \right|_{h}^{\frac{q_1\alpha _1}{q}})\times \cdot \cdot \cdot \times L^{q_m}(\mathbb{H} ^n,\left| x \right|_{h}^{\frac{q_m\alpha _m}{q}})\rightarrow L^q(\mathbb{H} ^n,\left| x \right|_{h}^{\alpha})}
		\\
		&=\int_0^{\infty}{\cdots \int_0^{\infty}{\int_{\mathbb{S} ^{Q-1}}{\cdots \int_{\mathbb{S} ^{Q-1}}{\frac{\Phi \left( \delta _{r_1}y_{1}^{\prime},...,\delta _{r_m}y_{m}^{\prime} \right)}{\left| r_1 \right|_{h}^{Q}\cdots \left| r_m \right|_{h}^{Q}}\prod_{j=1}^m{r_{j}^{\frac{\alpha _j}{q}+\frac{Q}{q_j}-\frac{\varepsilon}{q_j}-Q+1}}}}}}dy_{1}^{\prime}\cdots dy_{m}^{\prime}dr_1\cdots dr_m
		\\
		&=C_{\Phi}^{h}.
	\end{align*}
	This is similar to the proof of Theorem \ref{thm3.1}, so we omit the details. 
	This finishes the proof of Corollary \ref{cor5.1}.
\end{proof}\renewcommand{\qedsymbol}{}

\section{Conclusions}
First, in the setting of the Heisenberg group, the $n$-dimensional fractional Hardy operator has a sharp weak estimate from $L^p$ to $L^{q,\infty}$. The weak estimate bound is given by
\begin{align*}
	\left\| \mathcal{H}_{\alpha} \right\| _{L^p( \mathbb{H} ^n,\left| x \right|_{h}^{\beta} ) \rightarrow L^{q,\infty}( \mathbb{H} ^n,\left| x \right|_{h}^{\gamma} )}=\left( \frac{\omega _Q}{Q+\gamma} \right) ^{\frac{1}{q}}\left( \frac{\omega _Q\left( p-1 \right)}{pQ-Q-\beta} \right) ^{\frac{1}{p^{\prime}}}.
\end{align*}
Additionally, for the $L^1$ case, we have
\begin{align*}\left\| \mathcal{H}_{\alpha} \right\| _{L^1\left( \mathbb{H} ^n \right) \rightarrow L^{\left( Q+\beta \right) /\left( Q-\alpha \right) ,\infty}( \mathbb{H} ^n,\left| x \right|_{h}^{\gamma} )}=\left( \frac{\omega _Q}{Q+\beta} \right) ^{\left( Q-\alpha \right) /\left( Q+\beta \right)}.
\end{align*}

Second, we derive the sharp bounds for the $m$-linear $n$-dimensional integral operator with a kernel on weighted Lebesgue spaces: 
\begin{align*}
	\left\| \mathcal{H}^h\right\| _{L^{q_1}( \mathbb{H} ^n,\left| x \right|_{h}^{\frac{q_1\alpha _1}{q}} ) \times \cdot \cdot \cdot \times L^{q_m}( \mathbb{H} ^n,\left| x \right|_{h}^{\frac{q_m\alpha _m}{q}} ) \rightarrow L^q( \mathbb{H} ^n,\left| x \right|_{h}^{\alpha} )}=\int_{\mathbb{H} ^n}{\cdots \int_{\mathbb{H} ^n}{K\left( y_1,...,y_m \right) \prod_{i=1}^m{\left| y_i \right|_{h}^{-\frac{\alpha _j}{q}-\frac{Q}{q_j}}}}}dy_1\cdots dy_m.
\end{align*}

Finally, as an application, the sharp bounds for Hardy, Hardy-Littlewood-P\'{o}lya, and Hilbert operators on weighted Lebesgue spaces are obtained. Moreover, we also find the estimate for the Hausdorff operator on weighted $L^p$ spaces:
\begin{align*}
	&\left\| \mathcal{H} _{\Phi}^{h} \right\| _{L^{q_1}(\mathbb{H} ^n,\left| x \right|_{h}^{\frac{q_1\alpha _1}{q}})\times \cdot \cdot \cdot \times L^{q_m}(\mathbb{H} ^n,\left| x \right|_{h}^{\frac{q_m\alpha _m}{q}})\rightarrow L^q(\mathbb{H} ^n,\left| x \right|_{h}^{\alpha})}\\
	&=\int_0^{\infty}{\cdots \int_0^{\infty}{\int_{\mathbb{S} ^{Q-1}}{\cdots \int_{\mathbb{S} ^{Q-1}}{\frac{\Phi \left( \delta _{r_1}y_{1}^{\prime},...,\delta _{r_m}y_{m}^{\prime} \right)}{\left| r_1 \right|_{h}^{Q}\cdots \left| r_m \right|_{h}^{Q}}\prod_{j=1}^m{r_{j}^{\frac{\alpha _j}{q}+\frac{Q}{q_j}-\frac{\varepsilon}{q_j}-Q+1}}}}}}dy_{1}^{\prime}\cdots dy_{m}^{\prime}dr_1\cdots dr_m.
\end{align*}

\section*{Author contributions}
Tianyang He: Primarily responsible for the overall conceptual design of the paper and the proposal of core innovative ideas.
Zhiwen Liu: Mainly in charge of the derivation and proof of mathematical formulas, including theoretical analysis and the implementation of theorem proofs.
All authors have read and agree to the published version of the manuscript.

\section*{Use of Generative-AI tools declaration}
The authors declare they have not used Artificial Intelligence (AI) tools in the creation of this article.
%AI tools used:
%How were the AI tools used? 
%Where in the article is the information located?

\section*{Acknowledgments}
This work was supported by the National Natural Science Foundation of China with code 12471461 and the Fundamental Research Funds for the Central Universities. The funders had no role in the study design, data collection and analysis, decision to publish, or preparation of the manuscript.

\section*{Conflict of interest}
The authors declare that they have no conflict of interest.


\begin{thebibliography}{999}
\addtolength{\itemsep}{-0.6ex}
\bibitem{BI}
W. G. Faris, Weak Lebesgue spaces and quantum mechanical binding, {\it Duke Math. J.}, \textbf{43} (1976), 365--373. \doilink{https://doi.org/10.1215/S0012-7094-76-04332-5}

\bibitem{Hilbert} %1
A. Benyi, C. T. Oh, Best constants for certain multilinear integral operators, {\it J. Inequal. Appl.}, 2006, 28582. \doilink{https://doi.org/10.1155/jia/2006/28582}

\bibitem{HLP} %2
T. Batbold, Y. Sawano, Sharp bounds for $m$-linear Hilbert-type operators on the weighted Morrey spaces, {\it Math. Inequal. Appl.}, \textbf{20} (2017), 263--283. \doilink{https://doi.org/10.7153/mia-20-20}

\bibitem{CH}
M. Christ, L. Grafakos, Best constants for two nonconvolution inequalities, {\it P. Am. Math. Soc.}, \textbf{123} (1995), 1687--1693. \doilink{https://doi.org/10.1090/s0002-9939-1995-1239796-6}

\bibitem{chu} %3
J. Y. Chu, Z. W. Fu, Q. Y. Wu, $L^p$ and BMO bounds for weighted Hardy operators on the Heisenberg group, {\it J. Inequal. Appl.}, 2016, 1--12. \doilink{https://doi.org/10.1186/s13660-016-1222-x}

\bibitem{cou} %4
T. Coulhon, D. M$\ddot{u}$ller, J. Zienkiewicz, About Riesz transforms on the Heisenberg groups, {\it Math. Ann.}, \textbf{305} (1996), 369--379. \doilink{https://doi.org/10.1007/bf01444227}

\bibitem{Hardy1} %5
Z. W. Fu, L. Grafakos, S. Z. Lu, F. Y. Zhao, Sharp bounds for $m$-linear Hardy and Hilbert operators, {\it Houston J. Math.}, \textbf{38} (2012), 225--243.

\bibitem{GA}
G. Gao, X. Hu, C. Zhang, Sharp weak estimates for Hardy-type operators, {\it Ann. Funct. Anal.}, \textbf{7} (2016), 421--433. \doilink{https://doi.org/10.1215/20088752-3605447}

\bibitem{Hardyjieshao4} %6
G. Gao, F. Y. Zhao, Sharp weak bounds for Hausdorff operators, {\it Anal. Math.}, \textbf{41} (2015), 163--173. \doilink{https://doi.org/10.1007/s10476-015-0204-4}

\bibitem{HLPtuiguang} %7
Q. J. He, M.  Q. Wei, D. Y. Yan, Sharp bound for generalized $m$-linear $n$-dimensional Hardy-Littlewood-P\'{o}lya operator, {\it Anal. Theor. Appl.}, \textbf{37} (2021), 1--14. \doilink{https://doi.org/10.4208/ata.OA-2020-0039}

\bibitem{Hardy}
G. H. Hardy, Note on a theorem of Hilbert, {\it Math. Z.}, \textbf{6} (1920), 314--317. \doilink{https://doi.org/10.1007/bf01199965}

\bibitem{kor} %8
A. Kor\'anyi, H. Reimann, Quasiconformal mappings on the Heisenberg group, {\it Invent. Math.}, \textbf{80} (1985), 309--338. \doilink{https://doi.org/10.1007/bf01388609}

\bibitem{Hardyjieshao1} %9
S. Z. Lu, D. D. Yan, F. Y. Zhao, Sharp bounds for Hardy type operators on higher-dimensional product spaces, {\it J. Inequal. Appl.}, \textbf{2013} (2013), 1--11. \doilink{https://doi.org/10.1186/1029-242x-2013-148}

\bibitem{LU}
S. Lu, Some recent progress of $n$-dimensional Hardy operators, {\it Adv. Math. China}, \textbf{42} (2013), 737--747.

\bibitem{fractional} %10
Y. Mizuta, A. Nekvinda, T. Shimomura, Optimal estimates for the fractional Hardy operator, {\it Stud. Math.}, \textbf{227} (2015), 1--19. \doilink{https://doi.org/10.4064/sm227-1-1}

\bibitem{lunwenji2} %11
S. Semmes,  An introduction to Heisenberg groups in analysis and geometry, {\it Notices of the AMS}, \textbf{50} (2003), 173--186. 

\bibitem{Hardy2} %12
S. Stević, Note on norm of an $m$-linear integral-type operator between weighted-type spaces, {\it Adv Differ. Equ.}, \textbf{2021} (2021), 1--10. \doilink{https://doi.org/10.1186/s13662-021-03346-4}

\bibitem{tha} %13
S. Thangavelu, {\it Harmonic analysis on the Heisenberg group}, MA: Birkhauser Boston, \textbf{159} (1998).

\bibitem{p-adicHardy} %15
Q. Y. Wu, Z. W. Fu, Sharp estimates of $m$-linear $p$-adic Hardy and Hardy-Littlewood-P\'{o}lya operators,
{\it Appl. Math.}, 2011, 1--20.

\bibitem{fraction} %15
H. X. Yu, J. F. Li, Sharp weak bounds for $n$-dimensional fractional Hardy operators, {\it Front. Math. China}, \textbf{13} (2018), 449--457. \doilink{https://doi.org/10.1007/s11464-018-0685-0}

\bibitem{Hardyjieshao2} %16
F. Y. Zhao, Z. W. Fu, S. Z. Lu, Endpoint estimates for $n$-dimensional Hardy operators and their~commutators, {\it Sci. China Math.}, \textbf{55} (2012), 1977--1990. \doilink{https://doi.org/10.1007/s11425-012-4465-0}

\bibitem{Hardyjieshao3} %17
F. Y. Zhao, S. Z. Lu, The best bound for $n$-dimensional fractional Hardy operators, {\it Math. Inequal. Appl.}, \textbf{18} (2015), 233--240. \doilink{https://doi.org/10.7153/mia-18-17}

\bibitem{zha} %18
G. Zhang, Q. Li, Q. Wu, The Weighted and estimates for fractional Hausdorff operators on the Heisenberg Group, {\it J. Funct. Space.}, 2020.

\bibitem{Hardylunwenji}V. A. Zorich, O. Paniagua, {\it Mathematical analysis II}, Berlin: Springer, 2016. \doilink{https://doi.org/10.5860/choice.42-0997b}

\bibitem{Hausdorff}X. S. Zhang, M. Q. Wei, D. Y. Yan, Sharp bound of Hausdorff operators on Morrey spaces with power weights, {\it J. Univ.  Chinese Acad. Sci.}, \textbf{38} (2021), 577.

\end{thebibliography}
\end{document}